\documentclass{article}
\pdfoutput=1

\usepackage[utf8]{inputenc}
\usepackage[a4paper,left=3cm,right=3cm,bottom=3.7cm,top=3.5cm]{geometry}
\usepackage{amsmath}
\usepackage{amssymb}
\usepackage{amsthm}
\usepackage{bm}
\usepackage{mathtools}
\usepackage{microtype}
\usepackage{xcolor}
\usepackage[normalem]{ulem} 
\usepackage[hidelinks]{hyperref}
\usepackage{graphicx}
\usepackage{subcaption}
\usepackage{booktabs}
\usepackage{xparse}
\usepackage[ruled,vlined]{algorithm2e}
\SetKwInput{KwInput}{Input}
\SetKwInput{KwOutput}{Output}
\SetKwInput{KwParameters}{Parameters}
\usepackage{pgfplots}
\pgfplotsset{compat=1.16}
\usepackage{siunitx}
\usepackage{placeins}
\usepackage{ulem}

\theoremstyle{plain}
\newtheorem{theorem}{Theorem}[section]
\newtheorem{lemma}[theorem]{Lemma}

\newtheorem{definition}{Definition}

\DeclareMathOperator{\ortho}{ortho}

\newcommand{\abs}[1]{\left\lvert #1 \right\rvert}

\title{\Large A stable rank-adaptive step-and-truncate finite volume method for Vlasov transport on domains with piecewise linear boundaries}
\author{
    Andr\'e Uschmajew\thanks{Institute of Mathematics \& Centre for Advanced Analytics and Predictive Sciences, University of Augsburg, 86159 Augsburg, Germany} \and
    Andreas Zeiser\thanks{Faculty 1: School of Engineering -- Energy and Information, HTW Berlin -- University of Applied
        Sciences, 12459 Berlin, Germany}
}

\date{}

\numberwithin{equation}{section}

\begin{document}

\maketitle

\begin{abstract}
We consider the numerical solution of the linear Vlasov transport equation on bounded spatial domains with inflow boundary conditions based on low-rank approximation. We combine a finite volume discretization with a rank-adaptive step-and-truncate scheme for the resulting matrix ODE. The spatial and velocity meshes may be unstructured, while suitable numerical fluxes retain a separated space--velocity representation. For homogeneous inflow, we show that the low-rank scheme inherits the $L_2$ stability and CFL restriction of the underlying full finite volume forward Euler method. In addition, the low-rank approximation error is bounded explicitly in terms of the truncation tolerances, avoiding the modeling error associated with tangent-space projections in dynamical low-rank approximation. Numerical experiments in 1d1v and 2d2v confirm the predicted error behavior. In 2d2v, the method handles an unstructured spatial mesh with nonzero inflow and a full tensor-product discretization of approximately $5.8\cdot10^{10}$ phase-space cells while the numerical rank is at most twelve.
\end{abstract}



\section{Introduction}

The Vlasov equation is a kinetic transport equation describing the evolution of a particle density $u=u(t,\bm x,\bm v)$ in phase space~\cite{dolbeault2002}. As a model equation, we consider in this work the linear Vlasov transport equation
\[
    \partial_t u
    +
    \bm v\cdot\nabla_{\bm x}u
    -
    \bm E(t,\bm x)\cdot\nabla_{\bm v}u
    =0,
    \qquad
    (\bm x,\bm v)\in\Omega^{(x)}\times\mathbb R^d,
\]
with a spatial domain $\Omega^{(x)} \subset \mathbb R^d$ on which the electric field $\bm E$ is prescribed. A major challenge for the numerical solution is the dimensionality of the phase space: already in three spatial dimensions ($d=3$), the particle density depends on six phase-space variables in addition to time. Low-rank matrix approximations provide a widely used approach for mitigating the dimensionality by exploiting approximate separability between space and velocity.

A prominent class of low-rank methods is dynamical low-rank approximation (DLRA), where the phase-space density is approximated in the form
\[
    u_r(t,\bm x,\bm v)
    =
    \sum_{i,j=1}^{r}
    X_i(t,\bm x) S_{ij}(t) V_j(t,\bm v),
\]
and the evolution equation is projected onto the tangent space of the manifold of fixed-rank functions. For kinetic equations, this approach was introduced in~\cite{einkemmer2018a} using the projector-splitting integrator for time integration and has subsequently been extended in several directions, including conservative and rank-adaptive formulations~\cite{einkemmer2019,einkemmer2021,einkemmer2023}. A complementary approach to DLRA is to first discretize the evolution equation in full phase space and subsequently approximate the resulting high-dimensional discrete solution with low rank. In Step-And-Truncate (SAT) methods, the action of a time step of the full discrete problem is itself represented or approximated in low-rank form and then followed by an additional truncation to desired accuracy. In contrast to DLRA, the evolution in SAT is not restricted to the tangent space of a prescribed low-rank manifold, and hence no corresponding modeling error is introduced~\cite{appelo2025}. Instead, the additional low-rank approximation error is controlled directly through the truncation tolerances. SAT approaches of this type have been investigated, for example, in~\cite{kieri2019,guo2022a,guo2024a,guo2024b}. We refer to the survey~\cite{einkemmer2024review} for a general overview on low-rank methods for time-dependent kinetic equations with a discussion of both DLR and SAT.

Most low-rank discretizations of Vlasov-type equations have been developed for Cartesian and, in particular, periodic spatial domains. In contrast, in this work we focus on bounded spatial domains $\Omega^{(x)}\subset\mathbb R^d$ with piecewise linear boundaries and impose inflow boundary conditions on the spatial boundary. The treatment of physical boundaries is considerably less straightforward since boundary conditions then have to be incorporated into the low-rank evolution. In previous work~\cite{uschmajew2024}, we considered DLRA formulations with inflow for spatial domains with piecewise linear boundary based on a suitable continuous formulation that incorporates the boundary condition in the operator. This led to Friedrichs-type evolution equations for the low-rank factors, which were subsequently discretized by stabilized continuous finite elements. However, the existence of weak solutions of these continuous low-rank DLRA subproblems was not addressed in~\cite{uschmajew2024}. The present work therefore adopts a different strategy: rather than incorporating geometry and boundary conditions into continuous low-rank evolution equations, we first discretize the original transport problem and subsequently apply low-rank approximation to the resulting finite-dimensional evolution equation using a rank-adaptive SAT version of the forward Euler method.

As for the discretization, discontinuous Galerkin (DG) methods provide a natural framework for transport problems and have also been combined with low-rank approximations. A DG discretization of conservative DLRA equations was considered in~\cite{uschmajew2025} for periodic spatial domains. In the SAT setting, the LoMaC approach introduced in~\cite{guo2024b} was subsequently extended to a nodal DG discretization in~\cite{guo2024}, including the treatment of inflow boundaries on tensorized computational grids. More recently, a nodal DG discretization in physical space has been combined with a rank-adaptive low-rank representation in velocity for the BGK equation in~\cite{galindo-olarte2026}. 

In this work, however, we deliberately start from a finite volume (FV) discretization; see, e.g.,~\cite{eymard2000} for a general introduction. The lowest-order setting of FV permits a particularly transparent stability and error analysis for forward Euler schemes while retaining the essential difficulties associated with unstructured spatial meshes and inflow boundary conditions. It therefore provides a suitable setting for an analysis of the low-rank method and a starting point for future extensions to higher-order DG discretizations. We refer to~\cite{bouche2005,boyer2012} for results and further references regarding the convergence of FV schemes for linear transport equations including boundary values.

Let us describe our low-rank FV approach in a little more detail. For the numerical discretization, the velocity domain is truncated to a bounded computational domain $\Omega^{(v)}\subset\mathbb R^d$. We discretize space and velocity separately by FV meshes and use their product to construct the phase-space discretization. Notably, the meshes in both $\Omega^{(x)}$ and $\Omega^{(v)}$ may be unstructured. Using $L_2$-normalized cell indicator functions as basis functions, the coefficient matrix $U\in\mathbb R^{n^{(x)}\times n^{(v)}}$ of a FV function~$u_h$ satisfies
\[
    \|u_h\|_{L_2(\Omega^{(x)}\times\Omega^{(v)})}
    =
    \|U\|_F .
\]
This identification of the $L_2$-norm with the Frobenius norm will be essential in the stability analysis.

A central issue in the low-rank formulation is the representation of the numerical fluxes on unstructured meshes, where the face normals vary from face to face and may destroy the simple space--velocity separation of the transport field. A separated representation is immediate, for example, for a global Lax--Friedrichs flux. For boundary faces, however, one has to use an upwind flux to model the in- and outflow. Although the absolute value of the normal transport velocity appearing in this flux is not separable in general, a piecewise linear boundary involves only finitely many distinct boundary-normal directions, so that the corresponding boundary operators can be represented as finite sums of tensor product terms. Similar constructions may be used for other numerical fluxes, provided that the dependence on the face normals can be represented by sufficiently short separated expansions.

Assuming in addition that the prescribed inflow data admit a separated representation of finite rank, the resulting FV equation can be written as a matrix ODE of the form
\[
    \dot U
    =
    F_h(t,U)
    \coloneqq
    \sum_{k=1}^{n_L} A_k(t) U B_k(t)^\top
    +
    \sum_{k=1}^{n_B} a_k(t)b_k(t)^\top .
\]
The first sum represents the discrete transport operator, while the second contains the contribution of the inhomogeneous inflow boundary data. Consequently, the action of the FV operator can be evaluated in low-rank form without assembling full phase-space operators. For the time integration, we use a rank-adaptive SAT forward Euler scheme. Starting from a low-rank representation
\[
    U^n = X^n S^n V^{n,\top},
\]
we first approximate the FV right-hand side $F_h(t^n,U^n)$ by a truncated singular value decomposition (SVD). The corresponding left and right singular spaces are used to enrich the current spatial and velocity bases. A forward Euler step is then performed in the resulting augmented tensor-product space, followed by a second SVD truncation to obtain the new low-rank iterate. The first truncation therefore controls the approximation of the full discrete update direction, while the second determines the rank of the updated solution. The basis augmentation strategy is related to rank-adaptive low-rank integrators such as~\cite{ceruti2022,appelo2025}.

An important consequence of the discretize-first construction is a particularly transparent stability mechanism. The intermediate update of the proposed scheme can be written as
\[
    \tilde U^{n+1}
    =
    P_{\widetilde X}
    \bigl(
        U^n+\tau F_h(t^n,U^n)
    \bigr)
    P_{\widetilde V},
\]
where $P_{\tilde X}$ and $P_{\tilde V}$ denote the $L_2$-orthogonal projections onto the enriched spatial and velocity spaces. Thus, the intermediate update is simply an orthogonal projection of the full forward Euler FV step, and the final low-rank iterate is obtained by a subsequent SVD truncation which in fact is also an orthogonal projection. Since both operations are hence nonexpansive in the Frobenius norm, the $L_2$ stability of the full forward Euler step carries over directly to the rank-adaptive scheme. In particular, for a sufficiently regular prescribed electric field and homogeneous inflow conditions, we prove $L_2$ stability under the same CFL restriction as for the underlying full-order forward Euler method. In this context, it can be noted that stability of DLRA integrators for hyperbolic problems has also been studied in~\cite{kusch2023}. There, the stability properties depend on the evolution equations for the low-rank factors and on the particular DLRA integrator; specifically, a discretize-first projector-splitting formulation may lead to instabilities, whereas starting from a continuous DLRA formulation can recover the classical CFL condition. In contrast, the SAT formulation considered in the present work does not introduce low-rank evolution equations and stability is inherited directly from the fully discrete FV Euler step. 

The SAT structure also provides direct control of the error introduced by the low-rank approximation. Comparing the rank-adaptive iterates $U^n$ with the iterates $\hat U^n$ of the full FV forward Euler method, we obtain in Theorem~\ref{thm:euler_difference} the error estimate
\[
    \|U^n-\hat U^n\|_F
    \le
    \delta
    +
    T\varepsilon_1
    +
    \frac{T}{\tau}\varepsilon_2 ,
\]
where $\delta$ denotes the error in the initial value, $\varepsilon_1$ is the tolerance used to approximate the FV right-hand side, and $\varepsilon_2$ is the tolerance of the rank truncation for iterates. Hence, in contrast to the modeling error arising in DLRA, all additional errors introduced by the low-rank approximation are explicitly controlled by prescribed truncation tolerances. In particular, if the initial approximation is exact or satisfies $\delta=O(\tau)$, the choices
\[
    \varepsilon_1=O(\tau),
    \qquad
    \varepsilon_2=O(\tau^2)
\]
ensure that the error introduced by the low-rank scheme compared to the full forward Euler method remains of the order $O(\tau)$. Our analysis therefore controls the error relative to the full forward Euler method without imposing an a priori modeling assumption of low-rank approximability, as used, for example, in~\cite[Theorem~2]{ceruti2022}. This error control does not, however, guarantee that the numerical
ranks remain small. In practice, rank bounds could of course be enforced in the truncation steps, potentially reducing the accuracy. However, in our experiments we only used the rank-adaptive version with prescribed accuracy which automatically led to sufficiently small ranks.

To conclude, the main contributions of this work are summarized as follows.
\begin{enumerate}
    \item
    We derive a tensorized FV discretization of linear Vlasov transport on bounded piecewise linear spatial domains. The spatial and velocity meshes may both be unstructured. We show how suitable numerical fluxes, including the possibility of inflow boundaries, can be represented in separated space--velocity form and thereby lead to a matrix ODE suitable for low-rank simulation. The electric field $\bm E$ is assumed to be prescribed, which allows us to concentrate on the interaction between physical boundaries, discretization, and low-rank approximation. An extension to self-consistent field models, in particular the Vlasov--Poisson system, is left for future work.

    \item
    We employ a rank-adaptive SAT time integrator (Algorithm~\ref{alg:step_truncate}) for the resulting matrix ODE. The proposed integrator first compresses the right-hand side, enriches the current row and column spaces by the resulting bases, performs a projected forward Euler step in the enlarged tensor-product space, and finally truncates the updated matrix. We show that, for homogeneous inflow, the integrator inherits the $L_2$ stability and the CFL restriction of the full FV forward Euler scheme (Theorem~\ref{thm:stability_low_rank}). Moreover, we derive an explicit bound for the additional error introduced by the low-rank approximation in terms of the truncation tolerances (Theorem~\ref{thm:euler_difference}), without introducing the modeling error associated with the tangent-space projection in DLRA.

    \item
    We validate the method numerically in one spatial and one velocity dimension (1d1v) as well as in two spatial and two velocity dimensions (2d2v), confirming the theoretical predictions. In an additional 2d2v mesh-refinement experiment beyond the scope of the analysis, we demonstrate the practical applicability of the method on an unstructured spatial mesh with nonzero inflow. On the finest mesh, the corresponding full tensor-product discretization would contain approximately $5.8\cdot 10^{10}$ cells, while the numerical ranks remain at most twelve.
\end{enumerate}

Let us also comment on the limitations of this work and possible directions for future research. First, we do not analyze convergence of the rank-adaptive SAT scheme with respect to the mesh sizes. Instead, our analysis compares the SAT approximation with the corresponding untruncated scheme on the same discretization. In principle, an $L_2$ error estimate for the underlying FV Euler scheme could therefore be combined with our results to obtain a mesh-dependent error estimate for the SAT scheme. We do not pursue this here, since existing convergence results most closely related to our setting, such as~\cite{bouche2005,boyer2012,merlet2008}, do not directly apply, for example because they consider more restrictive numerical fluxes or do not cover bounded domains with inflow boundary conditions. Bridging these differences would require a separate, more detailed analysis.

Second, we do not address conservation properties or the preservation of nonnegativity of the computed densities. Regarding the latter, our numerical experiments indicate that negative values remain small in magnitude. Very recently, a positivity-preserving correction scheme has been proposed for DLRA in~\cite{kormann2026}, including variants that additionally also preserve mass and momentum. More generally, several conservative low-rank integrators for Vlasov-type problems have been developed~\cite{einkemmer2024review}, including methods for DG discretization~\cite{uschmajew2025}. It would be interesting to investigate in future work whether such ideas can be incorporated into the present SAT framework.

The remainder of the paper is organized as follows. In Section~\ref{sec:FV} we derive the tensorized FV discretization and discuss the representation of the interior and boundary fluxes. Section~\ref{sec:low_rank} introduces the rank-adaptive SAT scheme. Its stability and error with respect to the full FV forward Euler solution are analyzed in Section~\ref{sec:analysis}. The numerical experiments in 1d1v and 2d2v are presented in Section~\ref{sec:experiments}, and conclusions are drawn in Section~\ref{sec:conclusion}.

\section{Finite volume semi-discretization}\label{sec:FV}

In this section, we derive a FV semi-discretization of the linear Vlasov equation introduced above. We refer to~\cite{eymard2000} for a general introduction to FV methods. Since the resulting semi-discrete equations will subsequently be approximated in low rank, particular attention is paid to preserving a separated representation with respect to the spatial and velocity variables.

For the numerical discretization, we truncate the unbounded velocity domain $\mathbb R^d$ to a bounded polytope $\Omega^{(v)}\subset\mathbb R^d$ and define the computational phase-space domain
\[
    \Omega = \Omega^{(x)}\times\Omega^{(v)}.
\]
On this domain, we consider the initial-boundary value problem
\begin{equation}\label{eq:vp}
\begin{aligned}
    \partial_t u
    + \bm v\cdot\nabla_{\bm x}u
    - \bm E(t,\bm x)\cdot\nabla_{\bm v}u
    &=0
    &&\text{in } (0,T)\times\Omega, \\
    u(0,\bm x,\bm v)
    &=u_0(\bm x,\bm v)
    &&\text{in } \Omega, \\
    u(t,\bm x,\bm v)
    &=g(t,\bm x,\bm v)
    &&\text{on } (0,T)\times\Gamma_x^-, \\
    u(t,\bm x,\bm v)
    &=0
    &&\text{on } (0,T)\times\Gamma_v^-,
\end{aligned}
\end{equation}
where
\[
    \Gamma_x^-
    =
    \left\{
        (\bm x,\bm v)\in
        \partial\Omega^{(x)}\times\Omega^{(v)}
        :
        \bm n^{(x)}\cdot\bm v < 0
    \right\}
\]
and
\[
    \Gamma_v^-(t)
    =
    \left\{
        (\bm x,\bm v)\in
        \Omega^{(x)}\times\partial\Omega^{(v)}
        :
        -\bm n^{(v)}\cdot\bm E(t,\bm x) < 0
    \right\}
\]
denote the spatial and velocity inflow boundaries, respectively. The velocity domain $\Omega^{(v)}$ is chosen sufficiently large such that the influence of the artificial velocity boundary remains negligible. We impose homogeneous inflow conditions on $\Gamma_v^-(t)$. For notational convenience, we extend the spatial inflow function $g$ to the velocity inflow boundary by setting $g=0$ on $\Gamma_v^-(t)$.

\subsubsection*{Finite volume spaces}

We construct the FV space on the phase space as a product of FV spaces on the spatial and velocity domains. Let $\mathcal T^{(x)}$ and $\mathcal T^{(v)}$ be FV meshes of $\Omega^{(x)}$ and $\Omega^{(v)}$, respectively, consisting of open polytopal cells. For $y\in\{x,v\}$, let
\[
    n^{(y)} = \lvert \mathcal T^{(y)} \rvert
\]
be the number of cells in $\mathcal T^{(y)}$ and define the corresponding FV space by
\[
    \mathcal V_h^{(y)}
    =
    \left\{
        w\in L_2(\Omega^{(y)})
        :
        w|_K\in\mathcal P^0(K)
        \quad\text{for all } K\in\mathcal T^{(y)}
    \right\},
\]
where $\mathcal P^0(K)$ denotes the space of constant functions on $K$. We use the $L_2$-normalized cell indicator functions
\[
    \varphi_i^{(y)}
    =
    \frac{1}{\sqrt{|K_i^{(y)}|}}
    \chi_{K_i^{(y)}},
    \qquad
    K_i^{(y)}\in\mathcal T^{(y)},\quad i=1,\ldots,n^{(y)},
\]
as an orthonormal basis of $\mathcal V_h^{(y)}$.

The phase-space FV space is then defined as the tensor-product space
\[
    \mathcal V_h
    =
    \mathcal V_h^{(x)}\otimes\mathcal V_h^{(v)}
    =
    \left\{
        u_h(\bm x,\bm v)
        =
        \sum_{i=1}^{n^{(x)}}
        \sum_{j=1}^{n^{(v)}}
        u_{ij}
        \varphi_i^{(x)}(\bm x)
        \varphi_j^{(v)}(\bm v)
        :
        u_{ij}\in\mathbb R
    \right\}.
\]
It coincides with the space of piecewise constant functions on the product mesh
\begin{equation}\label{eq:product_mesh}
    \mathcal T
    =
    \left\{
        K^{(x)}\times K^{(v)}
        :
        K^{(x)}\in\mathcal T^{(x)},
        \;
        K^{(v)}\in\mathcal T^{(v)}
    \right\}.
\end{equation}

We collect the coefficients of a function $u_h\in\mathcal V_h$ in the matrix
\begin{equation} \label{eq:coefficient_matrix}
    U=[u_{ij}]
    \in
    \mathbb R^{n^{(x)}\times n^{(v)}}.
\end{equation}
On a product cell $K=K_i^{(x)}\times K_j^{(v)}$, the FV function $u_h$ takes the constant value
\begin{equation*}\label{eq:function_values}
    u_h|_K
    =
    \frac{u_{ij}}
    {\sqrt{|K_i^{(x)}||K_j^{(v)}|}}
    =
    \frac{u_{ij}}{\sqrt{|K|}}.
\end{equation*}
By orthonormality of the basis functions,
\[
    \|u_h\|_{L_2(\Omega)}^2
    =
    \sum_{i=1}^{n^{(x)}}
    \sum_{j=1}^{n^{(v)}}
    u_{ij}^2
    =
    \|U\|_F^2.
\]
Here and throughout, $\|\cdot\|_F$ denotes the Frobenius norm. Thus, the $L_2$-norm of a FV function is represented exactly by the Frobenius norm of its coefficient matrix.

\subsubsection*{Cell balance equation}

On the phase space $\Omega = \Omega^{(x)} \times \Omega^{(v)}$ we introduce the phase-space variable and gradient
\[
    \bm z = (\bm x,\bm v) \in \Omega^{(x)} \times \Omega^{(v)},
    \qquad
    \nabla_{\bm z} = (\nabla_{\bm x},\nabla_{\bm v}),
\]
and write the Vlasov equation in~\eqref{eq:vp} in conservative form as
\begin{equation}\label{eq:vlasov conservative form}
    \partial_t u + \nabla_{\bm z}\cdot \bm f(t,u) = 0,
    \qquad
    \bm f(t,u) = \bm a(t,\bm z) u,
    \qquad
    \bm a(t,\bm z) =
    \begin{bmatrix}
        \bm v \\
        -\bm E(t,\bm x)
    \end{bmatrix}.
\end{equation}
Here, the phase-space velocity $\bm a(t,\bm z)$ is divergence-free, $\nabla_{\bm z}\cdot\bm a=0$, so that the conservative and nonconservative forms are equivalent. Integrating the conservative equation over a cell $K \in \mathcal T$ and applying the divergence theorem gives the cell balance equation
\begin{equation}\label{eq:fv_cell_balance}
    \frac{\mathrm d}{\mathrm dt}\int_K u\,\mathrm d\bm z
    +
    \sum_{e\subset \partial K}
    \int_e \bm n_{K,e}\cdot \bm f(t,\bm z, u)\,\mathrm ds
    =
    0,
    \qquad K\in\mathcal T,
\end{equation}
where $\bm n_{K,e}$ denotes the outward unit normal of $K$ on the face $e$.

For a discrete function $u_h \in \mathcal V_h$, we approximate the normal fluxes $\bm n_{K,e} \cdot \bm f(t,\bm z, u_h)$ in~\eqref{eq:fv_cell_balance} by numerical fluxes. We refer to~\cite{leveque2002} for standard FV fluxes for hyperbolic problems. First, consider an interior face $e=\partial K^-\cap \partial K^+$ between two cells $K^-, K^+ \in \mathcal T$ and let $\bm n_e = [\bm n_e^{(x)}, \bm n_e^{(v)}]^\top$ be the unit normal pointing from $K^-$ to $K^+$. For a function $u$, the corresponding traces on $e$ are denoted by
\[
    u^-(\bm z) = \lim_{\epsilon\to 0^+} u(\bm z-\epsilon \bm n_e), \qquad
    u^+(\bm z) = \lim_{\epsilon\to 0^+} u(\bm z+\epsilon \bm n_e), \qquad \bm z \in e.
\]
Note that for functions $u_h \in \mathcal V_h$ the traces $u^\pm$ are constant on $e$ and take the values of $u_h$ in $K^-$ and $K^+$, respectively. We define
\[
    [u]_e = u^- - u^+, \qquad  \{ \bm f \}_e = \frac 1 2 (\bm f^- + \bm f^+).
\]
Define the numerical flux by
\[
    \hat f_e(t, \bm z, u^-,u^+,\bm n_e) \coloneqq \bm n_e\cdot \{ \bm a(t,\bm z)\, u \}_e
        + \frac{\alpha_e(t,\bm z)}{2}[u]_e.
\]
For the upwind flux $\hat f_e^{\mathrm{up}}$ we define
\begin{equation} \label{eq:upwind_flux}
    \alpha_e^{\mathrm{up}}(t, \bm z) = \abs{\bm n_e\cdot \bm a(t,\bm z)}.    
\end{equation}
For the local Lax--Friedrichs (Rusanov) flux $\hat f_e^{\mathrm{LF}}$ we choose a face-wise constant $\alpha_e^{\mathrm{LF}}$ such that
\begin{equation} \label{eq:lf_flux}
    \alpha_e^{\mathrm{LF}}(t) \ge \max_{\bm z\in e} \abs{\bm n_e\cdot \bm a(t,\bm z)}.
\end{equation}
For a global Lax--Friedrichs flux, we choose a parameter $\alpha^{\mathrm{gLF}}(t)$ such that 
\[ 
    \alpha^{\mathrm{gLF}}(t) \ge \max_{e \in \mathcal E} \max_{\bm z\in e} \abs{\bm n_e\cdot\bm a(t,\bm z)},
\]  
where $\mathcal E$ denotes the set of faces on which the global Lax--Friedrichs flux is applied. In the sequel, $\alpha_e$ will denote the corresponding dissipation parameter, i.e., $\alpha_e \in \{\alpha_e^{\mathrm{up}},\alpha_e^{\mathrm{LF}},\alpha^{\mathrm{gLF}}\}$ depending on the selected flux.

Second, for boundary faces $e\subset \partial K \cap \partial \Omega$ with outward unit normal $\bm n_e$, we prescribe the exterior trace by $u^+ = g$ and define the numerical flux by
\[
    \hat f_e(t, \bm z, u^-,g,\bm n_e) := \hat f_e^{\mathrm{up}}(t, \bm z, u^-,g,\bm n_e)
\]
to impose the boundary conditions.

The different numerical fluxes differ significantly with respect to their suitability for tensorization. Since the cells are polytopes, their faces are planar, and hence the components $\bm n_e^{(x)}$ and $\bm n_e^{(v)}$ are constant on each face and do not depend on the phase-space variables. Therefore,
\[
    \bm n_e\cdot \bm a(t,\bm z) =
        \bm n_e^{(x)}\cdot \bm v - \bm n_e^{(v)}\cdot \bm E(t,\bm x)
\]
is a sum of terms depending separately on $\bm x$ and $\bm v$. Consequently, the global Lax–Friedrichs flux naturally admits a separated representation. Similarly, a local Lax--Friedrichs flux can be applied on groups of phase-space cells, leading to a correspondingly larger number of tensor-product terms. 

In contrast, the upwind flux involves the coefficient $\abs{\bm n_e\cdot\bm a(t,\bm z)}$, whose absolute value generally destroys the additive separation. Nevertheless, the upwind flux can also be represented in tensorized form, but it involves a separate tensor-product term for each distinct face normal. Consequently, the number of terms grows with the number of distinct normal directions. For the boundary this number is finite and determined by the geometry of the domain, enabling us to use the upwind flux on boundary faces in a tensorized representation. We derive such a decomposition for boundary faces below. Employing upwind fluxes also for inner faces is limited by the fact that the number of distinct normal directions can be extremely high for unstructured meshes. However, for structured meshes, the number of distinct normal directions is small and the upwind flux can be used on interior faces as well.

\subsubsection*{Imposing inflow conditions}

As both the spatial and velocity boundaries are piecewise planar, we can write
\begin{equation}\label{eq: piecewise linear boudary}
    \partial\Omega^{(y)} = \bigcup_{\mu=1}^{n^{(y)}_\partial} \Gamma_\mu^{(y)},  \qquad y \in \{x,v\},
\end{equation}
with a constant outward normal $\bm n^{(y,\mu)}$ on each piece $\Gamma_\mu^{(y)}$. Due to the product structure of the phase space, the normals on the boundary faces belonging to $\Gamma_\mu^{(x)}$ and $\Gamma_\mu^{(v)}$ can be represented as $\bm n = (\bm n^{(x,\mu)}, \bm 0)$ and $\bm n = (\bm 0, \bm n^{(v,\mu)})$, respectively. Hence, on the spatial boundary
\[
    \hat f_{e}^{\mathrm{up}}\bigl(t,\bm z, u^-,g,(\bm n^{(x,\mu)},\bm 0)\bigr) =
    \begin{cases}
        (\bm n^{(x,\mu)}\cdot \bm v)\,u^-,
            & \bm n^{(x,\mu)}\cdot \bm v \ge 0, \\[0.2em]
        (\bm n^{(x,\mu)}\cdot \bm v)\,g,
            & \bm n^{(x,\mu)}\cdot \bm v < 0
    \end{cases}
\]
and on velocity boundary faces we have
\[
    \hat f_{e}^{\mathrm{up}}\bigl(t,\bm z, u^-,g,(\bm 0,\bm n^{(v,\mu)})\bigr) =
    \begin{cases}
        \bigl(-\bm n^{(v,\mu)}\cdot \bm E(t,\bm x)\bigr)\,u^-,
            & -\bm n^{(v,\mu)}\cdot \bm E(t,\bm x) \ge 0, \\[0.2em]
        0,
            & -\bm n^{(v,\mu)}\cdot \bm E(t,\bm x) < 0
    \end{cases}
\]
since homogeneous inflow boundary conditions are prescribed on the velocity inflow boundary.

\subsubsection*{Resulting semi-discrete equations}

Using the discussed numerical fluxes in \eqref{eq:fv_cell_balance}, we derive the semidiscrete equations for the coefficient matrix defined in \eqref{eq:coefficient_matrix}. To obtain a tensorized representation, we assume that the boundary data admits the separated representation
\begin{equation}\label{eq: separated g}
    g(t,\bm x, \bm v) = \sum_{\nu=1}^{n_g} g^{(x)}_\nu(t,\bm x) \, g^{(v)}_\nu(t, \bm v).
\end{equation}
Of course, in practice one might assume this only approximately. This would introduce an additional approximation error which is not discussed in our work.

For simplicity of presentation, we write out the semidiscrete equations explicitly only for the global Lax--Friedrichs flux on interior faces and the upwind flux on boundary faces. In this case we obtain the matrix ordinary differential equation
\begin{equation}
    \label{eq:ODE_FV}
    \begin{aligned}
        & \dot U 
        + \sum_{s=1}^d \left[ C^{(x)}_s U M^{(v),\top}_{v_s} +  M^{(x)}_{-E_s} U C^{(v),\top}_s \right] 
            + \frac{\alpha^{\mathrm{gLF}}(t)}{2} \left[D^{(x)} U  + U D^{(v),\top} \right] \\
        &\quad+
        \sum_{\mu=1}^{n_\partial^{(x)}}\sum_{s=1}^d B^{(x)}_{\mu,s} U M^{(v),\top}_{\chi^{(x)}_\mu v_s}
            + \sum_{\mu=1}^{n_\partial^{(v)}}\sum_{s=1}^d M^{(x)}_{-\chi^{(v)}_\mu E_s} U  B^{(v),\top}_{\mu,s} 
                + \sum_{\nu=1}^{n_g} \sum_{\mu=1}^{n_\partial^{(x)}}\sum_{s=1}^d G^{(x)}_{\nu,\mu,s} G^{(v),\top}_{\nu,\mu,s}
            = 0.
    \end{aligned}
\end{equation}
The matrices $M^{(x)}_{-E_s}$, $M^{(x)}_{-\chi^{(v)}_\mu E_s}$, $G^{(x)}_{\nu,\mu,s}$, and $G^{(v)}_{\nu,\mu,s}$ are time dependent through the electric field and the inflow data. The matrices and vectors for $y\in\{x,v\}$ are defined by
\[
\begin{alignedat}{2}
    [M^{(y)}_b]_{ij}
        &=
        \sum_{K\in \mathcal T^{(y)}}
        \int_{K} b \, \varphi^{(y)}_j \varphi^{(y)}_i\,\mathrm d\bm y,
    \qquad
    &
    [C^{(y)}_s]_{ij}
        &=
        \sum_{e\in\mathcal E^{(y)}_0}
        \int_e
            n^{(y)}_{e,s}\,
            \{\varphi_j^{(y)}\}_e
            [\varphi_i^{(y)}]_e
        \,\mathrm dS,
    \\[0.4em]
    [D^{(y)}]_{ij}
        &=
        \sum_{e\in\mathcal E^{(y)}_0}
        \int_e
            [\varphi_j^{(y)}]_e
            [\varphi_i^{(y)}]_e
        \,\mathrm dS,
    \qquad
    &
    [B^{(y)}_{\mu,s}]_{ij}
        &=
        \sum_{e \in \mathcal E^{(y)}_{b,\mu}}
        \int_e
            n_s^{(y,\mu)} \varphi_j^{(y)} \varphi_i^{(y)}
        \,\mathrm dS,
    \\[0.4em]
    [G^{(x)}_{\nu,\mu,s}]_{i}
        &=
        \sum_{e \in \mathcal E^{(x)}_{b,\mu}}
        \int_e
            n_s^{(x,\mu)}
            g^{(x)}_\nu \varphi_i^{(x)}
        \,\mathrm dS,
    \qquad
    &
    [G^{(v)}_{\nu,\mu,s}]_{i}
        &=
        \sum_{K\in \mathcal T^{(v)}}
        \int_{K}
            (1-\chi^{(x)}_\mu)\,
            v_s\,
            g^{(v)}_\nu \varphi^{(v)}_i
        \,\mathrm d\bm v .
\end{alignedat}
\]
Here, $b$ is an arbitrary function, and we denote by $\mathcal E^{(y)}_{0}$ and $\mathcal E^{(y)}_b$ the set of interior and boundary faces of $\mathcal{T}^{(y)}$, respectively. Let $\mathcal E^{(y)}_{b,\mu} = \{ e \in \mathcal E^{(y)}_b \colon e \subset \Gamma_\mu^{(y)}\}$ be the set of boundary faces on the $\mu$-th piece of the boundary. The spatial and velocity outflow indicators are defined as
\[
    \chi^{(x)}_\mu(\bm v) =
        \begin{cases}
            1, & \bm n^{(x,\mu)}\cdot \bm v \ge 0, \\[0.2em]
            0, & \bm n^{(x,\mu)}\cdot \bm v < 0.
        \end{cases}, \quad
    \chi^{(v)}_\mu(t,\bm x) =
        \begin{cases}
            1, & -\bm n^{(v,\mu)}\cdot \bm E(t,\bm x) \ge 0, \\[0.2em]
            0, & -\bm n^{(v,\mu)}\cdot \bm E(t,\bm x) < 0.
        \end{cases}
\]

The corresponding formulas for other fluxes discussed above are analogous but considerably more cumbersome.

In summary, using the product structure of the phase-space domain, a tensor product basis, suitable numerical fluxes, and a separated representation of the inflow data, the semidiscrete FV equations can be written in tensorized form as a sum of matrix products and low-rank boundary terms. This structure will be used in the next section to construct a low-rank integrator for the semidiscrete equations.

\section{Low-rank integrator \label{sec:low_rank}}

We formulate a time-stepping scheme for the semi-discrete equations~\eqref{eq:ODE_FV} based on low-rank approximations of the coefficient matrix $U(t)$. The proposed method can be seen as a low-rank version of the classical forward Euler method.

For convenience, we formulate our scheme for the matrix ODE
\begin{equation}
    \label{eq:matrix_ODE}
    \dot U = F_h(t,U) \coloneqq L_h(t,U) + B_h(t),
\end{equation}
where the affine-linear right-hand side can be written as
\begin{equation*}
        L_h(t,U) =\sum_{k=1}^{n_L} A_k(t) U B_k^\top(t), \qquad
        B_h(t) = \sum_{\ell=1}^{n_B} a_\ell(t) b_\ell^{\top}(t).
\end{equation*}
Here, 
\[ 
    U = U(t) \in \mathbb R^{n^{(x)} \times n^{(v)}}, \qquad
    A_k(t) \in \mathbb R^{n^{(x)} \times n^{(x)}}, \,
    B_k(t) \in \mathbb R^{n^{(v)} \times n^{(v)}}, \qquad
    a_\ell(t) \in \mathbb R^{n^{(x)}}, \, b_\ell(t) \in \mathbb R^{n^{(v)}}
\]
for $k=1,\ldots,n_L$ and $\ell=1,\ldots,n_B$. The semi-discrete equations~\eqref{eq:ODE_FV} obtained from our FV discretization are precisely of this form.

We first formally introduce the forward Euler method with constant time step $\tau$, as it will serve as a reference method for our low-rank integrator.

\begin{definition} \label{def:euler_fv}
    Let $N_T\in\mathbb N$ and define the uniform time discretization
    \[
        \tau = \frac{T}{N_T}, \qquad t^n = n\tau, \qquad n=0,\ldots,N_T .
    \]
    Let $\hat U^0$ be an initial matrix. Define the \emph{fully discrete forward Euler scheme} by
    \begin{equation*}\label{eq:euler_fv}
            \hat U^{n+1} = \hat U^n + \tau F_h(t^n, \hat U^n), \qquad n=0,\ldots,N_T-1.
    \end{equation*}
\end{definition}
Reasonable choices for $\hat U^0$ are the coefficients of the $L_2$ projection of the initial condition $u_0$ on the FV space $\mathcal V_h$ or an interpolation of $u_0$.

We next present the low-rank integrator, which will approximate the solution at times $t^n$ by a factorized low-rank representation
\begin{equation} \label{eq:U_low_rank}
    U(t^n) \approx U^n = X^n S^n V^{n,\top}, \qquad
    X^n \in \mathbb R^{n^{(x)} \times r^n}, \quad
    S^n \in \mathbb R^{r^n \times r^n}, \quad
    V^n \in \mathbb R^{n^{(v)} \times r^n},
\end{equation}
which corresponds to separation of spatial and velocity variables. The basic ingredient is the truncated singular value decomposition $\mathcal T^{\mathrm{svd}}_{\varepsilon}$ for factorized matrices, which we state for completeness in Algorithm~\ref{alg:truncated_SVD}. Importantly, using the factorized structure of the input matrix with small rank~$r$, Algorithm~\ref{alg:truncated_SVD} only requires computing the SVD of a small intermediate $r\times r$ matrix, which is then used to construct the truncated SVD of the full matrix. For numerical robustness, we will use column-pivoted QR decompositions for the initial orthogonalization step. We note that the main property we will exploit in the stability analysis of our SAT scheme (Theorem~\ref{thm:stability_low_rank}) is the nonexpansiveness of the SVD truncation in Frobenius norm, $\| \mathcal T^{\mathrm{svd}}_{\varepsilon} (M) \|_F \le \| M \|_F$, which is therefore explicitly stated in Algorithm~\ref{alg:truncated_SVD}. In principle, Algorithm~\ref{alg:truncated_SVD} could be replaced by other low-rank approximation methods with this property, such as randomized projection methods, provided they come with some reasonable error bounds~\cite{halko2011}.

\begin{algorithm}[t]
\DontPrintSemicolon
\LinesNumbered
\caption{Truncated SVD $\mathcal T^{\mathrm{svd}}_{\varepsilon}(M)$ for factorized matrices $M = A S B^\top$}
\label{alg:truncated_SVD}

\KwInput{Factorized matrix $M = A  S  B^\top$, $A \in \mathbb R^{n \times r}$, $ S\in \mathbb R^{r \times r}$, $B \in \mathbb R^{m \times r}$} 
\KwOutput{Truncated SVD $\mathcal T^{\mathrm{svd}}_{\varepsilon}\!(M) = U \Sigma V^\top$ satisfying
\[
    U^\top U = I, \quad V^\top V = I, \quad \| U \Sigma V^\top - M \|_{F} \le \varepsilon, \quad \| U \Sigma V^\top \|_{F} \le \| M \|_F
\]
}
\KwParameters{Tolerance $\varepsilon$}

Perform reduced column-pivoted QR decompositions:
\[
    Q_1 R_1 P_1^\top = A, \qquad Q_2 R_2 P_2^\top = B.
\]
\;

Form the reduced matrix
\[
    \tilde S = R_1 P_1^\top S P_2 R_2^\top.
\]
\;

Compute a truncated SVD $\tilde U\Sigma\tilde V^\top$ of $\tilde S$ such that
\[
\tilde U^\top \tilde U = I, \quad \tilde V^\top \tilde V = I, \quad
\|
    \tilde U\Sigma\tilde V^\top-\tilde S
    \|_F
    \le\varepsilon, \quad \| \Sigma \|_F \le \| \tilde S \|_F.
\]
\;

Return $U=Q_1\tilde U$, $V=Q_2\tilde V$, $\Sigma$.
\end{algorithm}

Note that Algorithm~\ref{alg:truncated_SVD} can also be used to compute a truncated SVD of a sum of matrices in factorized form
\[
    M = \sum_{k=1}^K \tilde A_k \tilde S_k \tilde B_k^\top
\]
by letting
\[
    A = [\tilde A_1, \ldots, \tilde A_K], \quad
    S = \mathrm{diag}(\tilde S_1, \ldots, \tilde S_K), \quad
    B = [\tilde B_1, \ldots, \tilde B_K]
\]
in the algorithm. In particular, the evaluation of $F_h$ at time $t^n$ and a matrix $U^n$ in the factorized form~\eqref{eq:U_low_rank} is given by
\[
    F_h(t^n, U^n) = L_h(t^n,U^n) + B_h(t^n)
        = \sum_{k=1}^{n_L} (A_k(t^n) X^n) S^n (B_k(t^n) V^n)^\top + \sum_{\ell=1}^{n_B} a_\ell(t^n) b_\ell^{\top}(t^n)
\]
and can therefore be compressed by a truncated SVD using Algorithm~\ref{alg:truncated_SVD}.

Using these ingredients, we can formulate one step of the proposed rank-adaptive integrator; see Algorithm~\ref{alg:step_truncate}. For a given $U^n$ in the form~\eqref{eq:U_low_rank}, we first compute the truncated SVD of the right-hand side. The resulting bases $\bar X$ and $\bar V$ resolve the right-hand side up to the prescribed tolerance $\varepsilon_1$. Augmenting these bases with the current bases $X^n$ and $V^n$ ensures that the current iterate $U^n$ can still be represented exactly in the enlarged bases $\tilde X$ and $\tilde V$. These enlarged bases are used to compute an update of the coefficient matrix $\tilde S^{n+1}$ in a Galerkin step. Finally, the resulting matrix is truncated again with tolerance $\varepsilon_2$ to obtain the low-rank iterate $U^{n+1}$.

The method belongs to the class of SAT schemes, in which a time step is performed in the ambient space and the result is subsequently compressed to low rank; cf.~\cite[Sect.~3]{einkemmer2024review}. For $\varepsilon_1=0$, the algorithm reduces to a rank-adaptive variant of the idealized projection method of~\cite{kieri2019} (for a forward Euler step). More specifically, the algorithm follows the same basic structure as the Merge method from~\cite{appelo2025}: the right-hand side $F_h$ is approximated in low-rank form, the bases are augmented, a Galerkin step is performed in the enlarged spaces, and the result is subsequently compressed. In~\cite{appelo2025}, however, this update is carried out implicitly and the basis enrichment additionally involves bases from so called $K$- and $L$-steps.

\begin{algorithm}[t]
\LinesNumbered
\DontPrintSemicolon
\caption{Rank-adaptive SAT integrator $\Phi_\tau(t^n,U^n; \varepsilon_1, \varepsilon_2)$}
\label{alg:step_truncate}

\KwInput{$U^n=X^n S^n V^{n,\top} \in \mathbb{R}^{n^{(x)} \times n^{(v)}}$ in SVD form, at time $t^n$}
\KwOutput{$U^{n+1}=X^{n+1}S^{n+1}V^{n+1,\top} \in \mathbb{R}^{n^{(x)} \times n^{(v)}}$ in SVD form}
\KwParameters{Step size $\tau$, tolerances $\varepsilon_1, \varepsilon_2$}
\BlankLine
Compute a truncated SVD of the right-hand side:
\[
    [\bar X, \bar S, \bar V]
        = \mathcal T^{\mathrm{svd}}_{\varepsilon_1} \big(F_h(t^n, U^n)\big).
\] \;
Compute orthonormal bases of the augmented spaces:
\[ 
    \tilde X = \ortho[X^n, \bar X], \qquad 
    \tilde V = \ortho[V^n, \bar V].
\] \;
Galerkin step with $\tilde S^n = (\tilde X^\top X^n) S^n (\tilde V^\top V^n)^\top$:
\[ 
    \tilde S^{n+1} = \tilde S^n 
        + \tau\, \tilde X^\top \, F_h(t^n,U^n) \, \tilde V.
\] \;
Rank truncation:
\[ 
    [X^{n+1}, S^{n+1}, V^{n+1}] 
    = \mathcal T^{\mathrm{svd}}_{\varepsilon_2}
        \bigl(\tilde X \tilde S^{n+1} \tilde V^\top\bigr).
\] \;
\end{algorithm}

Iterating these steps, we obtain our low-rank integrator for the semi-discrete equations~\eqref{eq:matrix_ODE}.

\begin{definition} \label{def:euler_low_rank}
    Let the assumptions of Definition~\ref{def:euler_fv} hold and let $\varepsilon_1,\varepsilon_2\ge 0$. Assume that we are given an initial matrix $U^0$ in SVD form. The iteration defined by
    \begin{equation}
        \label{eq:low_rank_fv}
            U^{n+1} = \Phi_\tau(t^n, U^n; \varepsilon_1, \varepsilon_2), \qquad n=0,\ldots,N_T-1
    \end{equation}
    using Algorithm~\ref{alg:step_truncate} for computing $\Phi_\tau$ is called the \emph{rank-adaptive SAT scheme}.
\end{definition}

In order to estimate the computational complexity, note that the discretization matrices occurring in the right-hand side are sparse, with a uniformly bounded number of nonzero entries per row and column. Their application therefore requires only linear work in the spatial or velocity dimension. Assuming that the ranks of the iterates are bounded by $r$, the factorized right-hand side has rank at most $n_L r+n_B$, and the dimension of the augmented bases is bounded by
\[
    r_*=(n_L+1)r+n_B.
\]
The computational complexity of one step of the low-rank integrator is therefore
\[
    \mathcal O\bigl(r_*^3+r_*^2(n^{(x)}+n^{(v)})\bigr).
\]
In contrast, the cost of one step of the full forward Euler scheme is
\[
    \mathcal O\bigl(n_L n^{(x)}n^{(v)}\bigr).
\]
Hence, if $n_L, n_B, r\ll\min\{n^{(x)},n^{(v)}\}$, the low-rank integrator is significantly cheaper than the full forward Euler scheme.

\section{Stability and conditional convergence} \label{sec:analysis}

In this section, we derive the main theoretical results for the proposed rank-adaptive SAT integrator applied to the FV discretization of the Vlasov equation.  We show that the method inherits the $L_2$ stability condition of the fully discrete forward Euler scheme and derive a bound on the difference to the full forward Euler method that explicitly accounts for the truncation errors.

As a first step, we state a stability result for the fully discrete forward Euler scheme with sufficiently small time steps governed by a CFL condition. We simplify the discussion by assuming homogeneous boundary conditions, that is, zero inflow. The general case can be treated by including the contribution of the inflow data, but requires additional technicalities. However, since the stability estimate will later be applied to the difference of two discrete solutions, for which the common inflow contribution cancels, the homogeneous version is sufficient for our purposes.

\begin{lemma}\label{lemma:stability_fe_vlasov}
    Consider the Vlasov equation~\eqref{eq:vp} on a domain with piecewise planar boundary~\eqref{eq: piecewise linear boudary} and homogeneous inflow conditions, that is, $g=0$. Assume the mesh $\mathcal T$ as defined in~\eqref{eq:product_mesh} and the FV discretization based on the numerical fluxes described in Section~\ref{sec:FV}, with uniformly bounded parameters $\alpha_e$. Let $\tau_{\mathrm{CFL}}\in(0,\infty]$ be the CFL threshold defined in~\eqref{eq:CFL}. Then, for every initial value $\hat U^0$ and for every time step $\tau$ with $0<\tau\le\tau_{\mathrm{CFL}}$, the fully discrete forward Euler scheme from Definition~\ref{def:euler_fv} is $L_2$ stable in the sense that its iterates $(\hat U^n)$ satisfy
    \[
        \|\hat U^n\|_F\le\|\hat U^0\|_F
    \]
    for all $n$ with $t^n=n\tau\le T$.
\end{lemma}

The phase-space velocity $\bm a$ in the conservative form~\eqref{eq:vlasov conservative form} is divergence-free. Consequently, the assertion, including the explicit CFL threshold in~\eqref{eq:CFL}, follows directly from the more generally Lemma~\ref{lemma:stability_fe} in the appendix.

We now proceed to prove the $L_2$ stability of the rank-adaptive SAT integrator in Algorithm~\ref{alg:step_truncate} under the same CFL condition as in Lemma~\ref{lemma:stability_fe_vlasov}. In fact, the key observation is that projections and rank truncations do not compromise the norm stability of the forward Euler scheme.

\begin{theorem}[Stability of rank-adaptive SAT integrator] \label{thm:stability_low_rank}
    Under the same assumptions as in Lemma~\ref{lemma:stability_fe_vlasov}, let $(U^n)_{n=0}^{N_T}$ be the sequence generated by the rank-adaptive SAT scheme~\eqref{eq:low_rank_fv} of Definition~\ref{def:euler_low_rank}. Then for every initial value $U^0$ and for every time step $\tau$ with $0<\tau \le \tau_{\mathrm{CFL}}$ it holds that
    \[
        \| U^n \|_F \le \| U^0 \|_F
    \]
    for all $n$ with $t^n = n \tau \le T$.
\end{theorem}

\begin{proof}
    We analyze one step of Algorithm~\ref{alg:step_truncate}. Observe that the intermediate quantity
    \[
        \tilde U^{n+1} = \tilde X \tilde S^{n+1} \tilde V^\top
    \]
    is just the orthogonal projection of a full forward Euler step onto the tensor product of spaces spanned by the orthonormal columns of \(\tilde X\) and \(\tilde V\), respectively. Indeed, by definition of $\tilde S^{n+1}$ and $\tilde S^n$ we have
    \begin{align}
    \tilde U^{n+1} &= \tilde X [\tilde S^n  + \tau\, \tilde X^\top \, F_h(t^n,U^n) \, \tilde V ] \tilde V^\top \notag \\
    &=\tilde X \tilde X^\top [ X^n S^n  (V^n)^\top + \tau \, F_h(t^n,U^n) ] \tilde V \tilde V^\top \notag \\
    &= \tilde X \tilde X^\top [ U^n + \tau \, F_h(t^n,U^n) ] \tilde V \tilde V^\top. \label{eq:intermediate step}
    \end{align}
    By Lemma~\ref{lemma:stability_fe_vlasov}, the forward Euler step is nonexpansive, so we get
    \[
        \| \tilde U^{n+1} \|_F \le \| U^n + \tau \, F_h(t^n,U^n) \|_F \le \| U^n \|_F.
    \] 
    The iterate $U^{n+1}$ is obtained by a truncated SVD of $\tilde U^{n+1}$ which is also nonexpansive in the Frobenius norm. Hence,
    \[
        \| U^{n+1} \|_F \le \|\tilde U^{n+1}\|_F \le \|U^n\|_F
    \]
    and the overall assertion follows by induction.
\end{proof}

We can now compare the rank-adaptive SAT scheme with the fully discrete forward Euler scheme in the presence of inhomogeneous inflow.

\begin{theorem} \label{thm:euler_difference}
    Assume the conditions of Theorem~\ref{thm:stability_low_rank}, but allow for nonhomogeneous inflow data~$g$. Let $(U^n)$ and $(\hat U^n)$ be the sequences generated by the rank-adaptive SAT scheme (Definition~\ref{def:euler_low_rank}) and the fully discrete forward Euler scheme (Definition~\ref{def:euler_fv}), respectively, using the same inflow data and the same time step $0<\tau\le\tau_{\mathrm{CFL}}$. Assume
    \[
        \|U^0-\hat U^0\|_F\le\delta
    \]
    for some $\delta \ge 0$. Then, 
    \[
        \|U^n-\hat U^n\|_F \le \delta+t^n\varepsilon_1 +\frac{t^n}{\tau}\varepsilon_2
        \le \delta+T\varepsilon_1 +\frac{T}{\tau}\varepsilon_2
    \]
    for all $n$ with $t^n=n\tau\le T$.
\end{theorem}
\begin{proof}
    We consider the step from $t^n$ to $t^{n+1}$ of the rank-adaptive SAT scheme. By~\eqref{eq:intermediate step}, the intermediate update before the $\varepsilon_2$-truncation satisfies
    \[
        \tilde U^{n+1} = U^{n} + \tau P_{\tilde X} F_h(t^n, U^n) P_{\tilde V}
    \]
    where $P_{\tilde X} = \tilde X \tilde X^\top$ and $P_{\tilde V} = \tilde V \tilde V^\top$. Here we used that the enlarged spaces contain the column spaces of $X^n$ and $V^n$, and hence $P_{\tilde X}U^nP_{\tilde V}=U^n$. Since $\bar X\bar S\bar V^\top = \mathcal T^{\mathrm{svd}}_{\varepsilon_1} \big(F_h(t^n, U^n)\big)$ belongs to the enlarged tensor-product space, the best-approximation property of the orthogonal projection yields
    \[
        \| P_{\tilde X} F_h(t^n, U^n) P_{\tilde V} - F_h(t^n, U^n)\|_F \le \| \mathcal T^{\mathrm{svd}}_{\varepsilon_1} \big(F_h(t^n, U^n)\big) - F_h(t^n, U^n)\|_F  \le \varepsilon_1.
    \]
    It follows that 
    \begin{align*}
            \| U^{n+1} - \hat U^{n+1} \|_F
            & \le \| U^{n+1} - \tilde U^{n+1} \|_F  + \| \tilde U^{n+1} - \hat U^{n+1} \|_F \\
            & \le \varepsilon_2 + \| U^n  + \tau P_{\tilde X} F_h(t^n,U^n) P_{\tilde V} - \hat U^n - \tau F_h(t^n, \hat U^n)\|_F \\
            & \le \varepsilon_2 + \tau \varepsilon_1 + \| U^n + \tau F_h(t^n,U^n) - \hat U^n  - \tau F_h(t^n,\hat U^n)\|_F\\
            & = \varepsilon_2 + \tau \varepsilon_1 + \| (U^n - \hat U^n )  + \tau L_h(t^n,U^n - \hat U^n) \|_F.
    \end{align*}
    In the last equality, the affine terms $B_h(t^n)$ related to the boundary inflow cancel since both schemes use the same inflow data. Hence the last term inside the norm is a fully discrete forward Euler step for homogeneous inflow conditions, which by Lemma~\ref{lemma:stability_fe_vlasov} satisfies
    \[
     \| (U^n - \hat U^n )  + \tau L_h(t^n,U^n - \hat U^n) \|_F \le \| U^n - \hat U^n \|_F.
    \]
    We conclude
    \[
        \| U^{n+1} - \hat U^{n+1} \|_F
        \le \| U^{n} - \hat U^{n} \|_F + \tau \varepsilon_1 + \varepsilon_2.
    \]
    Iterating this estimate and using $\|U^0-\hat U^0\|_F\le\delta$ proves the assertion.
\end{proof}

Theorem~\ref{thm:euler_difference} suggests a heuristic how the truncation tolerances should scale with the time step in order for the rank-adaptive SAT scheme to retain the first-order accuracy of the fully discrete forward Euler scheme. Indeed, assume that for all admissible time steps $0<\tau \leq \tau_{\mathrm{CFL}}$, the full forward Euler scheme satisfies
\[
    \|U(t^n)-\hat U^n\|_F \leq C \tau
\]
for all $n$ with $t^n = n \tau \leq T$, where $C$ is independent of $\tau$. Here, $U(t)$ denotes the exact solution of the semi-discrete matrix ODE~\eqref{eq:ODE_FV}, that is, the coefficient matrix of the semi-discrete FV solution. Suppose further that
\[
    \|U^0-\hat U^0\|_F \leq \frac{C}{3} \tau.
\]
Then, choosing
\[
    \varepsilon_1 = \frac{C}{3T}\tau, \qquad \varepsilon_2 = \frac{C}{3T}\tau^2
\]
Theorem~\ref{thm:euler_difference} yields
\[
    \|U^n-\hat U^n\|_F\le C\tau.
\]
Consequently,
\[
    \|U(t^n)-U^n\|_F  \le \|U(t^n)-\hat U^n\|_F  +  \|\hat U^n-U^n\|_F \le 2C\tau.
\]

Clearly, these observations only concern the accuracy of the rank-adaptive scheme and do not guarantee whether ``low-rank'' numerical solutions are obtained when using this heuristic. As demonstrated in the numerical experiments this may still be expected in several cases.

\section{Numerical experiments}\label{sec:experiments}

We present numerical experiments to demonstrate the performance of the proposed rank-adaptive SAT integrator and illustrate the theoretical convergence result from the previous section. In a first experiment we will investigate how the truncation tolerances influence the accuracy of the method. In order to be able to compare the results with a fully discrete forward Euler integrator, we restrict ourselves to the case of a one-dimensional spatial and velocity domain (1d1v). As a second experiment we then consider a two-dimensional triangular spatial domain with inflow and a corresponding two-dimensional velocity domain (2d2v). Here, we discretize the spatial domain with unstructured meshes to showcase the geometric flexibility of the method. Since the number of unknowns exceeds $5.8\cdot 10^{10}$ phase-space cells for the finest refinement level, we no longer compare the low-rank solution with the full FV solution. Instead, we estimate the $L_2$ error using a low-rank approximation of the exact solution projected onto the same mesh with piecewise-linear elements.

Our implementation is written in C++ and uses the finite element library MFEM~\cite{mfem2024} and standard numerical routines. All experiments have been performed on a desktop computer (Mac mini M4, 16 GB RAM).

\subsection{Free transport on an interval (1d1v) \label{sub:1d1v}}

We conduct our first experiment in a one-dimensional spatial and velocity domain. The purpose of this experiment is to verify several properties of the rank-adaptive SAT integrator in a setting where the corresponding full forward Euler solution can be computed directly. In particular, we investigate consistency in the absence of truncation, the $L_2$ stability of the method, the influence of the truncation tolerances $\varepsilon_1$ and $\varepsilon_2$, and the scaling of the SAT error with the time-step size.

We consider the free-transport equation
\[
    \partial_t u + v\,\partial_x u = 0 \quad \textnormal{on } \quad \Omega^{(x)}\times\Omega^{(v)}=[-0.5,1.5]\times[1,3]
\]
for $t\in[0,0.2]$, with homogeneous inflow conditions. The initial condition is a smooth bump centered at $(x_0,v_0)=(0,2)$ with width $(\sigma_x,\sigma_v)=(0.2,0.5)$:
\begin{equation}\label{eq: bump function}
    u_0(x,v) = \Psi\left(\frac{x-x_0}{\sigma_x}\right) \Psi\left(\frac{v-v_0}{\sigma_v}\right), \qquad  \Psi(z)=
    \begin{cases}
        (1-z^2)^4, & \abs{z}<1,\\
        0,         & \abs{z}\geq 1.
    \end{cases}
\end{equation}

The domain is discretized using uniform meshes with $n^{(x)}=160$ and $n^{(v)}=128$ cells in the spatial and velocity domains, respectively, and use the time step $\tau=5\cdot10^{-4}$. The domain is sufficiently large that the support of the solution does not reach the phase-space boundary during the simulation. Hence, boundary inflow and outflow do not affect the experiment. For the interior faces, we use a groupwise Lax--Friedrichs flux: on all faces associated with the spatial direction we choose $\alpha_e^{\mathrm{LF}}=3$, whereas on all faces associated with the velocity direction we set $\alpha_e^{\mathrm{LF}}=0$, since the electric field vanishes. Thus, the Lax--Friedrichs parameter is constant within each group of interior faces. As throughout the paper, the upwind flux is used on boundary faces.

For the chosen discretization, the CFL threshold defined in \eqref{eq:CFL} is \(\tau_{\mathrm{CFL}}\approx 4.17 \cdot 10^{-3}\). Hence, the time step \(\tau=5\cdot10^{-4}\) corresponds to the normalized CFL number
\[
\nu_{\mathrm{CFL}}=\frac{\tau}{\tau_{\mathrm{CFL}}}\approx 0.12 <1.
\]

We first verify that the SAT scheme introduces no additional approximation error in the absence of truncation. To this end, we set both prescribed truncation tolerances to zero and compare the SAT iterates with the corresponding full forward Euler iterates, using the same initial data, spatial discretization, and time step. The difference is evaluated at every time step in the Frobenius norm. While no tolerance-based truncation is applied, we still truncate to the numerical rank by only keeping singular values satisfying
\[
\sigma_j > \epsilon_{\mathrm{mach}}\, s\, \sigma_1,
\qquad
s=\min(n,m),
\]
for an \(n\times m\) matrix, where \(\sigma_1\) denotes the largest singular value and $\epsilon_{\mathrm{mach}}$ is the machine precision. With \(\varepsilon_1=\varepsilon_2=0\), the maximum difference between SAT and full forward Euler iterates over all times was approximately \(1.2\cdot 10^{-13}\), confirming agreement up to accumulated floating-point errors.

To study the influence of the tolerances $\varepsilon_1$ and $\varepsilon_2$ on the accuracy of the rank-adaptive SAT integrator, we next perform three sets of simulations. First, we vary the truncation tolerance $\varepsilon_1$ of the right-hand side while keeping the final truncation tolerance fixed at $\varepsilon_2=10^{-13}$. Second, we vary $\varepsilon_2$ and fix $\varepsilon_1=10^{-7}$. Finally, we choose
\begin{equation}
    \label{eq:epsilon_C}
    \varepsilon_1=C\tau, \qquad \varepsilon_2=C\tau^2
\end{equation}
and vary the tolerance factor $C$. We compute the maximum $L_2$ difference between the rank-adaptive SAT solution and the full forward Euler solution over all time steps. The results are shown in Figure~\ref{fig:interval-tolerance-study}. In general, in all three sets of simulations, the measured difference between SAT and forward Euler remains below the corresponding theoretical upper bound from Theorem~\ref{thm:euler_difference}. Let us discuss the observed dependence on the tolerances in more detail.

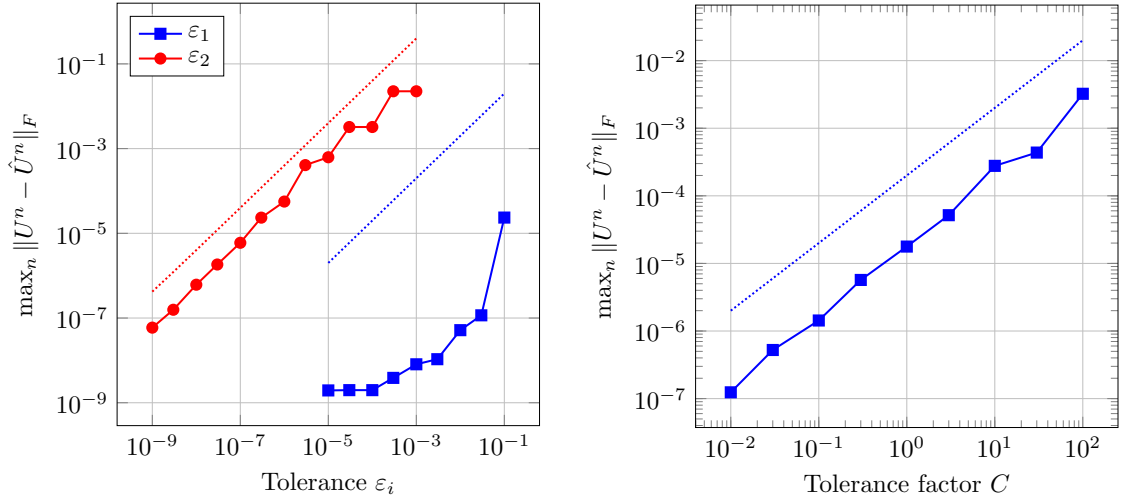
\begin{figure}[t]
    \centering

    \begin{minipage}{.49\textwidth}
\pgfplotstableread{
epsilon_1 epsilon_1_error epsilon_1_bound epsilon_1_id epsilon_2 epsilon_2_error epsilon_2_bound epsilon_2_id
1.0000000000000001e-05 1.9552196334471304e-09 2.0000400000000002e-06 702 1.0000000000000001e-09 5.9335708604876633e-08 4.2000000000000006e-07 600
3.0000000000000001e-05 1.9846787498336158e-09 6.0000400000000004e-06 703 3.0000000000000000e-09 1.5681741463279879e-07 1.2200000000000000e-06 601
1.0000000000000000e-04 1.9846787498336158e-09 2.0000040000000001e-05 704 1.0000000000000000e-08 6.1206322399666580e-07 4.0199999999999996e-06 602
2.9999999999999997e-04 3.8646654722804442e-09 6.0000039999999998e-05 705 2.9999999999999997e-08 1.8357722775342768e-06 1.2019999999999998e-05 603
1.0000000000000000e-03 8.0879743759266117e-09 2.0000004000000001e-04 706 9.9999999999999995e-08 5.9772505402361088e-06 4.0020000000000000e-05 604
3.0000000000000001e-03 1.0670905876468194e-08 6.0000004000000000e-04 707 2.9999999999999999e-07 2.3450342926032901e-05 1.2001999999999999e-04 605
1.0000000000000000e-02 5.1713688311196162e-08 2.0000000399999999e-03 708 9.9999999999999995e-07 5.6203302109508095e-05 4.0001999999999995e-04 606
2.9999999999999999e-02 1.1590685429106026e-07 6.0000000400000000e-03 709 3.0000000000000001e-06 4.0776212661907345e-04 1.2000200000000002e-03 607
1.0000000000000001e-01 2.3540687231144897e-05 2.0000000040000004e-02 710 1.0000000000000001e-05 6.2489129557350288e-04 4.0000200000000000e-03 608
nan nan nan nan 3.0000000000000001e-05 3.2340116927617960e-03 1.2000020000000000e-02 609
nan nan nan nan 1.0000000000000000e-04 3.2340116927617960e-03 4.0000020000000004e-02 610
nan nan nan nan 2.9999999999999997e-04 2.2559287684886412e-02 1.2000002000000000e-01 611
nan nan nan nan 1.0000000000000000e-03 2.2559287684886412e-02 4.0000002000000001e-01 612
}\intervalTolTolFData
\begin{tikzpicture}[scale=.95]
  \begin{loglogaxis}[
    scale only axis,
    width=0.8\columnwidth,
    height=0.8\columnwidth,
    xlabel={Tolerance $\varepsilon_i$},
    ylabel={$\max_n \| U^n - \hat U^n\|_F$},
    grid=major,
    legend pos=north west,
    legend cell align=left,
    unbounded coords=discard,
  ]
    \addplot[
      blue,
      thick,
      mark=square*,
      mark options={draw=blue, fill=blue},
    ] table[
      x=epsilon_1,
      y=epsilon_1_error,
    ] {\intervalTolTolFData};
    \addlegendentry{$\varepsilon_1$}

    \addplot[
      blue,
      densely dotted,
      thick,
      mark=none,
      forget plot,
    ] table[
      x=epsilon_1,
      y=epsilon_1_bound,
    ] {\intervalTolTolFData};

    \addplot[
      red,
      thick,
      mark=*,
      mark options={draw=red, fill=red},
    ] table[
      x=epsilon_2,
      y=epsilon_2_error,
    ] {\intervalTolTolFData};
    \addlegendentry{$\varepsilon_2$}

    \addplot[
      red,
      densely dotted,
      thick,
      mark=none,
      forget plot,
    ] table[
      x=epsilon_2,
      y=epsilon_2_bound,
    ] {\intervalTolTolFData};
  \end{loglogaxis}
\end{tikzpicture}

    \end{minipage}
     \hfill
    \begin{minipage}{.49\textwidth}
\pgfplotstableread{
C error bound id
1.0000000000000000e-02 1.2392328279580821e-07 1.9999999999999999e-06 850
2.9999999999999999e-02 5.2378599359791249e-07 6.0000000000000002e-06 851
1.0000000000000001e-01 1.4309867603410792e-06 2.0000000000000005e-05 852
2.9999999999999999e-01 5.6945866508184847e-06 6.0000000000000002e-05 853
1.0000000000000000e+00 1.7751134696198433e-05 2.0000000000000001e-04 854
3.0000000000000000e+00 5.1688352785440947e-05 6.0000000000000006e-04 855
1.0000000000000000e+01 2.7745838398643283e-04 2.0000000000000000e-03 856
3.0000000000000000e+01 4.3628086178547012e-04 6.0000000000000001e-03 857
1.0000000000000000e+02 3.2333394070751514e-03 2.0000000000000000e-02 858
}\intervalCData
\begin{tikzpicture}[scale=.95]
  \begin{loglogaxis}[
    scale only axis,
    width=0.8\columnwidth,
    height=0.8\columnwidth,
    xlabel={Tolerance factor $C$},
    ylabel={$\max_n \| U^n - \hat U^n\|_F$},
    grid=major,
    unbounded coords=discard,
  ]
    \addplot[
      blue,
      thick,
      mark=square*,
      mark options={draw=blue, fill=blue},
    ] table[
      x=C,
      y=error,
    ] {\intervalCData};

    \addplot[
      blue,
      densely dotted,
      thick,
      mark=none,
      forget plot,
    ] table[
      x=C,
      y=bound,
    ] {\intervalCData};
  \end{loglogaxis}
\end{tikzpicture}

    \end{minipage}
    \caption{Maximum $L_2$ difference over all time steps between the rank-adaptive SAT solution and the full-rank forward Euler solution for the 1d1v experiment. Left: Difference obtained by varying~$\varepsilon_1$ while fixing $\varepsilon_2=10^{-13}$ (blue), and by varying $\varepsilon_2$ while fixing $\varepsilon_1=10^{-7}$ (red). Right: Difference obtained by varying the tolerance factor $C$, where $\varepsilon_1=C\tau$ and $\varepsilon_2=C\tau^2$. The dotted lines indicate the upper bound from Theorem~\ref{thm:euler_difference} on the respective tolerance parameter.}
    \label{fig:interval-tolerance-study}
\end{figure}

In the left panel of Figure~\ref{fig:interval-tolerance-study}, the error obtained by varying $\varepsilon_2$ exhibits a clear first-order dependence on the tolerance, in agreement with the scaling predicted by the theorem. The dependence on $\varepsilon_1$ is less pronounced. For larger values of $\varepsilon_1$, the error decreases as the tolerance is reduced, but remains substantially below the corresponding upper bound. For sufficiently small $\varepsilon_1$, the error levels off at approximately $10^{-9}$, so that further decreasing $\varepsilon_1$ does not lead to a noticeable reduction of the SAT--forward Euler difference. Thus, the $\varepsilon_1$ contribution of the theoretical bound appears to be rather pessimistic in this experiment.

The right panel shows the results for the coupled tolerance choice~\eqref{eq:epsilon_C}. The error again remains below the theoretical bound $2TC\tau$ for all considered values of $C$ and exhibits an approximately first-order dependence on $C$. This behavior is consistent with the error scaling predicted by Theorem~\ref{thm:euler_difference} for the coupled tolerance choice~\eqref{eq:epsilon_C}. Table~\ref{tab:interval_C} complements these results with the maximal and average numerical ranks. Both decrease as $C$ increases, as expected from the more aggressive truncation.

\begin{table}[tb]
    \centering
    \caption{Comparison of the SAT and full forward Euler solutions for 1d1v experiment with the coupled tolerance choice \(\varepsilon_1=C\tau\), \(\varepsilon_2=C\tau^2\), for different values of \(C\). Reported are the maximum SAT--forward Euler difference over all time steps, the theoretical bound \(2TC\tau\) from Theorem~\ref{thm:euler_difference}, the maximal \(r_{\max}\) and average \(r_{\mathrm{avg}}\) SAT rank.}
    \label{tab:interval_C}
\begingroup
\small
\setlength{\tabcolsep}{4pt}
\sisetup{
  detect-all,
  output-exponent-marker=\ensuremath{\mathrm{E}},
}
\begin{tabular}{@{}
  S[table-format=3.2]
  S[table-format=1.2e-1]
  S[table-format=1.2e-1]
  S[table-format=2.0]
  S[table-format=2.2]
@{}}
  \toprule
  {$C$}
  & {$\max_n \|U^n-\hat U^n\|_F$} & {theor.\ bound} & {$r_{\max}$} & {$r_{\mathrm{avg}}$} \\
  \midrule
  0.01 & 1.24e-7 & 2.00e-6 & 7 & 5.86 \\
  0.03 & 5.24e-7 & 6.00e-6 & 7 & 5.39 \\
  0.1 & 1.43e-6 & 2.00e-5 & 6 & 5.02 \\
  0.3 & 5.69e-6 & 6.00e-5 & 5 & 4.52 \\
  1 & 1.78e-5 & 2.00e-4 & 5 & 4.14 \\
  3 & 5.17e-5 & 6.00e-4 & 4 & 3.67 \\
  10 & 2.77e-4 & 2.00e-3 & 4 & 3.13 \\
  30 & 4.36e-4 & 6.00e-3 & 3 & 2.79 \\
  100 & 3.23e-3 & 2.00e-2 & 2 & 2.00 \\
  \bottomrule
\end{tabular}
\endgroup

\end{table}

We next verify the $L_2$ stability established in Theorem~\ref{thm:stability_low_rank}. We consider a representative SAT computation with \(C=10\). The left panel of Figure~\ref{fig:interval-stability-dt} shows the normalized discrete $L_2$ norm
\(
    \frac{\|U^n\|_F}{\|U^0\|_F}
\)
for the SAT and full forward Euler solutions. Both curves nearly coincide, and in both cases the norm is nonincreasing over the complete time interval, in agreement with Theorem~\ref{thm:stability_low_rank}. The maximal observed stepwise change is negative for both solutions, confirming the nonexpansive behavior predicted by the theorem.

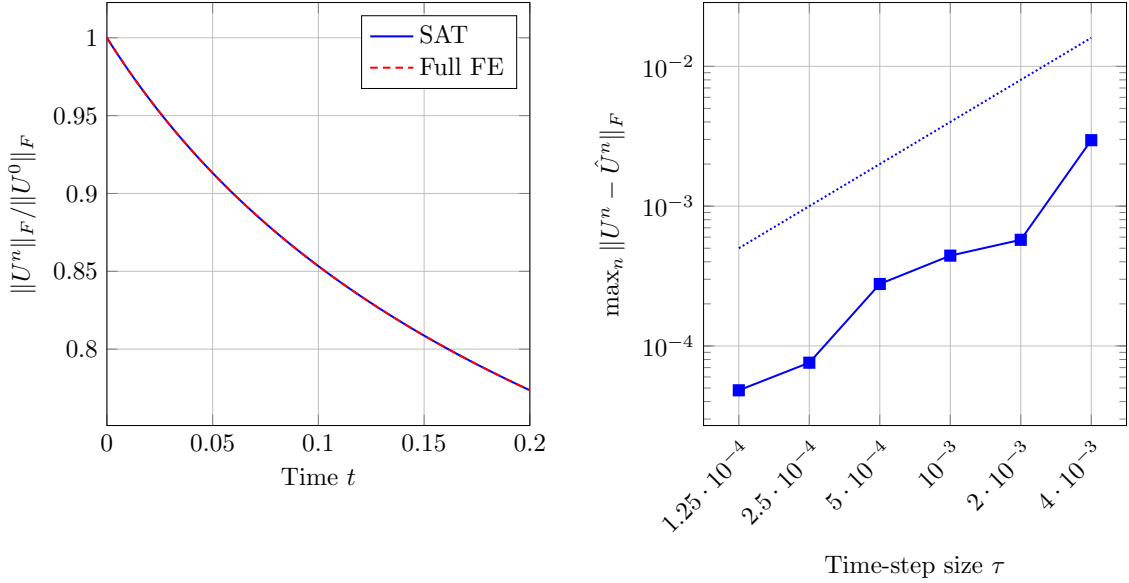
\begin{figure}[t]
    \centering

    \begin{minipage}[t]{.49\textwidth}
        \vspace{0pt}
        \centering
\pgfplotstableread{
time norm_sat norm_fe
0.0000000000000000e+00 1.0000000000000000e+00 1.0000000000000000e+00
5.0000000000000001e-04 9.9892765185953680e-01 9.9892765185953647e-01
1.0000000000000000e-03 9.9786051019177879e-01 9.9786051019200650e-01
1.5000000000000000e-03 9.9679853197306312e-01 9.9679853197374524e-01
2.0000000000000000e-03 9.9574167476898534e-01 9.9574167477034359e-01
2.5000000000000001e-03 9.9468989671325725e-01 9.9468989671551356e-01
3.0000000000000001e-03 9.9364315648897228e-01 9.9364315649234469e-01
3.5000000000000001e-03 9.9260141331177598e-01 9.9260141331648455e-01
4.0000000000000001e-03 9.9156462691456515e-01 9.9156462692082026e-01
4.5000000000000005e-03 9.9053275753344772e-01 9.9053275754143999e-01
5.0000000000000010e-03 9.8950576589477601e-01 9.8950576590467665e-01
5.4999999999999997e-03 9.8848361320303901e-01 9.8848361321500045e-01
6.0000000000000001e-03 9.8746626112945490e-01 9.8746626114374347e-01
6.5000000000000006e-03 9.8645367180153298e-01 9.8645367181846499e-01
7.0000000000000010e-03 9.8544580779299462e-01 9.8544580781292423e-01
7.4999999999999997e-03 9.8444263211426830e-01 9.8444263213759020e-01
8.0000000000000002e-03 9.8344410820345507e-01 9.8344410823060302e-01
8.5000000000000006e-03 9.8245019991773919e-01 9.8245019994917915e-01
9.0000000000000011e-03 9.8146087152516825e-01 9.8146087156140605e-01
9.5000000000000015e-03 9.8047608769679362e-01 9.8047608773836736e-01
1.0000000000000000e-02 9.7949581349913273e-01 9.7949581354661508e-01
1.0500000000000001e-02 9.7852001438691549e-01 9.7852001444091918e-01
1.1000000000000001e-02 9.7754865619613363e-01 9.7754865625730314e-01
1.1500000000000000e-02 9.7658170523165444e-01 9.7658170520633236e-01
1.2000000000000000e-02 9.7561912797674111e-01 9.7561912786662930e-01
1.2500000000000001e-02 9.7466089137189915e-01 9.7466089117861321e-01
1.3000000000000001e-02 9.7370696271332591e-01 9.7370696243845556e-01
1.3500000000000002e-02 9.7275730964710239e-01 9.7275730929220705e-01
1.4000000000000000e-02 9.7181190016350383e-01 9.7181189973012994e-01
1.4500000000000001e-02 9.7087070259150976e-01 9.7087070208118431e-01
1.5000000000000001e-02 9.6993368559342552e-01 9.6993368500767929e-01
1.5500000000000000e-02 9.6900081815969530e-01 9.6900081750007838e-01
1.6000000000000000e-02 9.6807206960384129e-01 9.6807206887195185e-01
1.6500000000000001e-02 9.6714740955752909e-01 9.6714740875505434e-01
1.7000000000000001e-02 9.6622680796577787e-01 9.6622680709454978e-01
1.7500000000000002e-02 9.6531023508231195e-01 9.6531023414434503e-01
1.8000000000000002e-02 9.6439766146503580e-01 9.6439766046254871e-01
1.8500000000000003e-02 9.6348905797168638e-01 9.6348905690704090e-01
1.9000000000000000e-02 9.6258439575557542e-01 9.6258439463115464e-01
1.9500000000000000e-02 9.6168364626143732e-01 9.6168364507944704e-01
2.0000000000000000e-02 9.6078678122123751e-01 9.6078677998358764e-01
2.0500000000000001e-02 9.5989377265000664e-01 9.5989377135832143e-01
2.1000000000000001e-02 9.5900459284181860e-01 9.5900459149754891e-01
2.1500000000000002e-02 9.5811921436592851e-01 9.5811921297045965e-01
2.2000000000000002e-02 9.5723761006308350e-01 9.5723760861777762e-01
2.2499999999999999e-02 9.5635975304187970e-01 9.5635975154807107e-01
2.3000000000000000e-02 9.5548561667513932e-01 9.5548561513415209e-01
2.3500000000000000e-02 9.5461517459620326e-01 9.5461517300953069e-01
2.4000000000000000e-02 9.5374840069468680e-01 9.5374839906497411e-01
2.4500000000000001e-02 9.5288526911388249e-01 9.5288526744509694e-01
2.5000000000000001e-02 9.5202575425085589e-01 9.5202575254505584e-01
2.5500000000000002e-02 9.5116983074886174e-01 9.5116982900727798e-01
2.6000000000000002e-02 9.5031747349442719e-01 9.5031747171826386e-01
2.6500000000000003e-02 9.4946865761502508e-01 9.4946865580547157e-01
2.7000000000000000e-02 9.4862335847560342e-01 9.4862335663422115e-01
2.7500000000000000e-02 9.4778155167414047e-01 9.4778154980468909e-01
2.8000000000000001e-02 9.4694321304284590e-01 9.4694321114893798e-01
2.8500000000000001e-02 9.4610831864527722e-01 9.4610831672801032e-01
2.9000000000000001e-02 9.4527684476861940e-01 9.4527684282908531e-01
2.9500000000000002e-02 9.4444876792338583e-01 9.4444876596264005e-01
3.0000000000000002e-02 9.4362406484052708e-01 9.4362406285972955e-01
3.0499999999999999e-02 9.4280271246746405e-01 9.4280271046925745e-01
3.1000000000000000e-02 9.4198468796396129e-01 9.4198468595531881e-01
3.1500000000000000e-02 9.4116996871253333e-01 9.4116996669458008e-01
3.2000000000000001e-02 9.4035853230003763e-01 9.4035853027371985e-01
3.2500000000000001e-02 9.3955035652063357e-01 9.3955035448687485e-01
3.3000000000000002e-02 9.3874541937346379e-01 9.3874541733318073e-01
3.3500000000000002e-02 9.3794369906016217e-01 9.3794369701430580e-01
3.4000000000000002e-02 9.3714517398092900e-01 9.3714517193204994e-01
3.4500000000000003e-02 9.3634982272886691e-01 9.3634982068597405e-01
3.5000000000000003e-02 9.3555762410707155e-01 9.3555762207107873e-01
3.5500000000000004e-02 9.3476855710379991e-01 9.3476855507549772e-01
3.6000000000000004e-02 9.3398260089809759e-01 9.3398259887824653e-01
3.6500000000000005e-02 9.3319973485765173e-01 9.3319973284700752e-01
3.6999999999999998e-02 9.3241993853661642e-01 9.3241993653593125e-01
3.7499999999999999e-02 9.3164319167343201e-01 9.3164318968348570e-01
3.7999999999999999e-02 9.3086947418884047e-01 9.3086947221033611e-01
3.8500000000000000e-02 9.3009876618362497e-01 9.3009876421725979e-01
3.9000000000000000e-02 9.2933104793662225e-01 9.2933104598307936e-01
3.9500000000000000e-02 9.2856629990269335e-01 9.2856629796264745e-01
4.0000000000000001e-02 9.2780450271072190e-01 9.2780450078483834e-01
4.0500000000000001e-02 9.2704563716165655e-01 9.2704563525058459e-01
4.1000000000000002e-02 9.2628968422656111e-01 9.2628968233094655e-01
4.1500000000000002e-02 9.2553662504471668e-01 9.2553662316518881e-01
4.2000000000000003e-02 9.2478644092172690e-01 9.2478643905890812e-01
4.2500000000000003e-02 9.2403911332766653e-01 9.2403911148217177e-01
4.3000000000000003e-02 9.2329462389525763e-01 9.2329462206769075e-01
4.3500000000000004e-02 9.2255295441806184e-01 9.2255295260901482e-01
4.4000000000000004e-02 9.2181408684870170e-01 9.2181408505876072e-01
4.4499999999999998e-02 9.2107800329711054e-01 9.2107800152684893e-01
4.4999999999999998e-02 9.2034468602880426e-01 9.2034468427878879e-01
4.5499999999999999e-02 9.1961411746317057e-01 9.1961411573396012e-01
4.5999999999999999e-02 9.1888628017180451e-01 9.1888627846394599e-01
4.6500000000000000e-02 9.1816115687683619e-01 9.1816115519086883e-01
4.7000000000000000e-02 9.1743873044929769e-01 9.1743872878575472e-01
4.7500000000000001e-02 9.1671898390751827e-01 9.1671898226692250e-01
4.8000000000000001e-02 9.1600190041552809e-01 9.1600189879839433e-01
4.8500000000000001e-02 9.1528746328148969e-01 9.1528746168832620e-01
4.9000000000000002e-02 9.1457565595615509e-01 9.1457565438745980e-01
4.9500000000000002e-02 9.1386646203132749e-01 9.1386646048759357e-01
5.0000000000000003e-02 9.1315986523836790e-01 9.1315986372007829e-01
5.0500000000000003e-02 9.1245584944669911e-01 9.1245584795432988e-01
5.1000000000000004e-02 9.1175439866233776e-01 9.1175439719635876e-01
5.1500000000000004e-02 9.1105549702645638e-01 9.1105549558732957e-01
5.2000000000000005e-02 9.1035912881394032e-01 9.1035912740211966e-01
5.2500000000000005e-02 9.0966527843198985e-01 9.0966527704792066e-01
5.2999999999999999e-02 9.0897393041871821e-01 9.0897392906284513e-01
5.3499999999999999e-02 9.0828506944178444e-01 9.0828506811454113e-01
5.3999999999999999e-02 9.0759868029703705e-01 9.0759867899884850e-01
5.4500000000000000e-02 9.0691474790716631e-01 9.0691474663845362e-01
5.5000000000000000e-02 9.0623325732039151e-01 9.0623325608157101e-01
5.5500000000000001e-02 9.0555419370915735e-01 9.0555419250063585e-01
5.6000000000000001e-02 9.0487754236883322e-01 9.0487754119101504e-01
5.6500000000000002e-02 9.0420328871645450e-01 9.0420328756973767e-01
5.7000000000000002e-02 9.0353141828946126e-01 9.0353141717423346e-01
5.7500000000000002e-02 9.0286191674445881e-01 9.0286191566110330e-01
5.8000000000000003e-02 9.0219476985598368e-01 9.0219476880488192e-01
5.8500000000000003e-02 9.0152996351531278e-01 9.0152996249683803e-01
5.9000000000000004e-02 9.0086748372925463e-01 9.0086748274377293e-01
5.9500000000000004e-02 9.0020731661897269e-01 9.0020731566684686e-01
6.0000000000000005e-02 8.9954944841882045e-01 8.9954944750040877e-01
6.0499999999999998e-02 8.9889386547519412e-01 8.9889386459085008e-01
6.0999999999999999e-02 8.9824055424539684e-01 8.9824055339546549e-01
6.1499999999999999e-02 8.9758950129650816e-01 8.9758950048133102e-01
6.2000000000000000e-02 8.9694069330427906e-01 8.9694069252419628e-01
6.2500000000000000e-02 8.9629411705204076e-01 8.9629411630738476e-01
6.3000000000000000e-02 8.9564975942962177e-01 8.9564975872071639e-01
6.3500000000000001e-02 8.9500760743227026e-01 8.9500760675944002e-01
6.4000000000000001e-02 8.9436764815960912e-01 8.9436764752317366e-01
6.4500000000000002e-02 8.9372986881458416e-01 8.9372986821485867e-01
6.5000000000000002e-02 8.9309425670244214e-01 8.9309425613973481e-01
6.5500000000000003e-02 8.9246079922970067e-01 8.9246079870431738e-01
6.6000000000000003e-02 8.9182948390315420e-01 8.9182948341539781e-01
6.6500000000000004e-02 8.9120029832887071e-01 8.9120029787903687e-01
6.7000000000000004e-02 8.9057323021121260e-01 8.9057322979959574e-01
6.7500000000000004e-02 8.8994826735186672e-01 8.8994826697875795e-01
6.8000000000000005e-02 8.8932539764888185e-01 8.8932539731456728e-01
6.8500000000000005e-02 8.8870460909572713e-01 8.8870460880048774e-01
6.9000000000000006e-02 8.8808588978034464e-01 8.8808588952446221e-01
6.9500000000000006e-02 8.8746922788423532e-01 8.8746922766798064e-01
7.0000000000000007e-02 8.8685461168153334e-01 8.8685461150517919e-01
7.0500000000000007e-02 8.8624202953810827e-01 8.8624202940192587e-01
7.1000000000000008e-02 8.8563146991066988e-01 8.8563146981492236e-01
7.1500000000000008e-02 8.8502292134588589e-01 8.8502292129083338e-01
7.2000000000000008e-02 8.8441637247950611e-01 8.8441637246540628e-01
7.2500000000000009e-02 8.8381181203549408e-01 8.8381181206260473e-01
7.2999999999999995e-02 8.8320922882518960e-01 8.8320922889376352e-01
7.3499999999999996e-02 8.8260861174644878e-01 8.8260861185673545e-01
7.3999999999999996e-02 8.8200994978282288e-01 8.8200994993507054e-01
7.4499999999999997e-02 8.8141323200272770e-01 8.8141323219717582e-01
7.4999999999999997e-02 8.8081844755862715e-01 8.8081844779551621e-01
7.5499999999999998e-02 8.8022558568623577e-01 8.8022558596580680e-01
7.5999999999999998e-02 8.7963463570371070e-01 8.7963463602619762e-01
7.6499999999999999e-02 8.7904558701087698e-01 8.7904558737651350e-01
7.6999999999999999e-02 8.7845842908844396e-01 8.7845842949745956e-01
7.7499999999999999e-02 8.7787315149723255e-01 8.7787315194985827e-01
7.8000000000000000e-02 8.7728974387742353e-01 8.7728974437388130e-01
7.8500000000000000e-02 8.7670819594779514e-01 8.7670819648830822e-01
7.9000000000000001e-02 8.7612849750498123e-01 8.7612849808977566e-01
7.9500000000000001e-02 8.7555063842275216e-01 8.7555063905204045e-01
8.0000000000000002e-02 8.7497460865125631e-01 8.7497460932525950e-01
8.0500000000000002e-02 8.7440039821633475e-01 8.7440039893526544e-01
8.1000000000000003e-02 8.7382799721878957e-01 8.7382799798285848e-01
8.1500000000000003e-02 8.7325739583367823e-01 8.7325739664310009e-01
8.2000000000000003e-02 8.7268858430964336e-01 8.7268858516462300e-01
8.2500000000000004e-02 8.7212155296819516e-01 8.7212155386893886e-01
8.3000000000000004e-02 8.7155629220304798e-01 8.7155629314975980e-01
8.3500000000000005e-02 8.7099279247944716e-01 8.7099279347233205e-01
8.4000000000000005e-02 8.7043104433350316e-01 8.7043104537276006e-01
8.4500000000000006e-02 8.6987103837152546e-01 8.6987103945735533e-01
8.5000000000000006e-02 8.6931276526938728e-01 8.6931276640198774e-01
8.5500000000000007e-02 8.6875621577187867e-01 8.6875621695144267e-01
8.6000000000000007e-02 8.6820138069205055e-01 8.6820138191877583e-01
8.6500000000000007e-02 8.6764825091061648e-01 8.6764825218469632e-01
8.7000000000000008e-02 8.6709681737531108e-01 8.6709681869693445e-01
8.7500000000000008e-02 8.6654707110027696e-01 8.6654707246963270e-01
8.8000000000000009e-02 8.6599900316545408e-01 8.6599900458273571e-01
8.8499999999999995e-02 8.6545260471599494e-01 8.6545260618138553e-01
8.8999999999999996e-02 8.6490786696164124e-01 8.6490786847532619e-01
8.9499999999999996e-02 8.6436478117615489e-01 8.6436478273831940e-01
8.9999999999999997e-02 8.6382333869672800e-01 8.6382334030755448e-01
9.0499999999999997e-02 8.6328353092340648e-01 8.6328353258307711e-01
9.0999999999999998e-02 8.6274534931851954e-01 8.6274535102721583e-01
9.1499999999999998e-02 8.6220878540612245e-01 8.6220878716402227e-01
9.1999999999999998e-02 8.6167383077142179e-01 8.6167383257870522e-01
9.2499999999999999e-02 8.6114047706024344e-01 8.6114047891708934e-01
9.2999999999999999e-02 8.6060871597847055e-01 8.6060871788505233e-01
9.3500000000000000e-02 8.6007853929151146e-01 8.6007854124800687e-01
9.4000000000000000e-02 8.5954993882376218e-01 8.5954994083034386e-01
9.4500000000000001e-02 8.5902290645807877e-01 8.5902290851491925e-01
9.5000000000000001e-02 8.5849743413524993e-01 8.5849743624252450e-01
9.5500000000000002e-02 8.5797351385348508e-01 8.5797351601136307e-01
9.6000000000000002e-02 8.5745113766789782e-01 8.5745113987655086e-01
9.6500000000000002e-02 8.5693029769000451e-01 8.5693029994960024e-01
9.7000000000000003e-02 8.5641098608720823e-01 8.5641098839791607e-01
9.7500000000000003e-02 8.5589319508232031e-01 8.5589319744430903e-01
9.8000000000000004e-02 8.5537691695306184e-01 8.5537691936649973e-01
9.8500000000000004e-02 8.5486214403157523e-01 8.5486214649662917e-01
9.9000000000000005e-02 8.5434886870394455e-01 8.5434887122078151e-01
9.9500000000000005e-02 8.5383708340972408e-01 8.5383708597850727e-01
1.0000000000000001e-01 8.5332678064146239e-01 8.5332678326235745e-01
1.0050000000000001e-01 8.5281795294423535e-01 8.5281795561740736e-01
1.0100000000000001e-01 8.5231059291519495e-01 8.5231059564080469e-01
1.0150000000000001e-01 8.5180469320309926e-01 8.5180469598131014e-01
1.0200000000000001e-01 8.5130024650786762e-01 8.5130024933884541e-01
1.0250000000000001e-01 8.5079724558014047e-01 8.5079724846404414e-01
1.0300000000000001e-01 8.5029568322082538e-01 8.5029568615781459e-01
1.0350000000000001e-01 8.4979555228066117e-01 8.4979555527090178e-01
1.0400000000000001e-01 8.4929684565979646e-01 8.4929684870344491e-01
1.0450000000000001e-01 8.4879955630734538e-01 8.4879955940456409e-01
1.0500000000000000e-01 8.4830367722097588e-01 8.4830368037192205e-01
1.0550000000000000e-01 8.4780920144647987e-01 8.4780920465131393e-01
1.0600000000000000e-01 8.4731612207736495e-01 8.4731612533624767e-01
1.0650000000000000e-01 8.4682443225444037e-01 8.4682443556752640e-01
1.0700000000000000e-01 8.4633412516540629e-01 8.4633412853285861e-01
1.0750000000000000e-01 8.4584519404445846e-01 8.4584519746643416e-01
1.0800000000000000e-01 8.4535763217188675e-01 8.4535763564853972e-01
1.0850000000000000e-01 8.4487143287367195e-01 8.4487143640516238e-01
1.0900000000000000e-01 8.4438658952110823e-01 8.4438659310759290e-01
1.0950000000000000e-01 8.4390309553040743e-01 8.4390309917204187e-01
1.1000000000000000e-01 8.4342094436231774e-01 8.4342094805926227e-01
1.1050000000000000e-01 8.4294012952174724e-01 8.4294013327415507e-01
1.1100000000000000e-01 8.4246064455738634e-01 8.4246064836541645e-01
1.1150000000000000e-01 8.4198248306133949e-01 8.4198248692514621e-01
1.1200000000000000e-01 8.4150563866874961e-01 8.4150564258848848e-01
1.1250000000000000e-01 8.4103010505744180e-01 8.4103010903327180e-01
1.1300000000000000e-01 8.4055587594756564e-01 8.4055587997964076e-01
1.1350000000000000e-01 8.4008294510122716e-01 8.4008294918970416e-01
1.1400000000000000e-01 8.3961130632214676e-01 8.3961131046717874e-01
1.1450000000000000e-01 8.3914095345530004e-01 8.3914095765704477e-01
1.1500000000000000e-01 8.3867188038657992e-01 8.3867188464519304e-01
1.1550000000000001e-01 8.3820408104245059e-01 8.3820408535808577e-01
1.1600000000000001e-01 8.3773754938960798e-01 8.3773755376242320e-01
1.1650000000000001e-01 8.3727227943464144e-01 8.3727228386479269e-01
1.1700000000000001e-01 8.3680826522371232e-01 8.3680826971135247e-01
1.1750000000000001e-01 8.3634550084220105e-01 8.3634550538748964e-01
1.1800000000000001e-01 8.3588398041441381e-01 8.3588398501750705e-01
1.1850000000000001e-01 8.3542369810323103e-01 8.3542370276428113e-01
1.1900000000000001e-01 8.3496464810980187e-01 8.3496465282896781e-01
1.1950000000000001e-01 8.3450682467322779e-01 8.3450682945066512e-01
1.2000000000000001e-01 8.3405022207024326e-01 8.3405022690610520e-01
1.2050000000000000e-01 8.3359483461490758e-01 8.3359483950935198e-01
1.2100000000000000e-01 8.3314065665829928e-01 8.3314066161148703e-01
1.2150000000000000e-01 8.3268768258821468e-01 8.3268768760030176e-01
1.2200000000000000e-01 8.3223590682886250e-01 8.3223591190000279e-01
1.2250000000000000e-01 8.3178532384056569e-01 8.3178532897091517e-01
1.2300000000000000e-01 8.3133592811946067e-01 8.3133593330917843e-01
1.2350000000000000e-01 8.3088771419722129e-01 8.3088771944646389e-01
1.2400000000000000e-01 8.3044067664074905e-01 8.3044068194967535e-01
1.2450000000000000e-01 8.2999481005190479e-01 8.2999481542067532e-01
1.2500000000000000e-01 8.2955010906721105e-01 8.2955011449598259e-01
1.2550000000000000e-01 8.2910656835757446e-01 8.2910657384650333e-01
1.2600000000000000e-01 8.2866418262801234e-01 8.2866418817725707e-01
1.2650000000000000e-01 8.2822294661736684e-01 8.2822295222708742e-01
1.2700000000000000e-01 8.2778285509804050e-01 8.2778286076839580e-01
1.2750000000000000e-01 8.2734390287572390e-01 8.2734390860687579e-01
1.2800000000000000e-01 8.2690608478912508e-01 8.2690609058123377e-01
1.2850000000000000e-01 8.2646939570970712e-01 8.2646940156292981e-01
1.2900000000000000e-01 8.2603383054141799e-01 8.2603383645591633e-01
1.2950000000000000e-01 8.2559938422043833e-01 8.2559939019637540e-01
1.3000000000000000e-01 8.2516605171492630e-01 8.2516605775245988e-01
1.3050000000000000e-01 8.2473382802474671e-01 8.2473383412404300e-01
1.3100000000000001e-01 8.2430270818123674e-01 8.2430271434245539e-01
1.3150000000000001e-01 8.2387268724694362e-01 8.2387269347024572e-01
1.3200000000000001e-01 8.2344376031537270e-01 8.2344376660092600e-01
1.3250000000000001e-01 8.2301592251076039e-01 8.2301592885872543e-01
1.3300000000000001e-01 8.2258916898780809e-01 8.2258917539834953e-01
1.3350000000000001e-01 8.2216349493145824e-01 8.2216350140473748e-01
1.3400000000000001e-01 8.2173889555664381e-01 8.2173890209282696e-01
1.3450000000000001e-01 8.2131536610805922e-01 8.2131537270731014e-01
1.3500000000000001e-01 8.2089290185991937e-01 8.2089290852240671e-01
1.3550000000000001e-01 8.2047149811573461e-01 8.2047150484162157e-01
1.3600000000000001e-01 8.2005115020807495e-01 8.2005115699752851e-01
1.3650000000000001e-01 8.1963185349834311e-01 8.1963186035153135e-01
1.3700000000000001e-01 8.1921360337655558e-01 8.1921361029364714e-01
1.3750000000000001e-01 8.1879639526110481e-01 8.1879640224226524e-01
1.3800000000000001e-01 8.1838022459855164e-01 8.1838023164394968e-01
1.3850000000000001e-01 8.1796508686339842e-01 8.1796509397320116e-01
1.3900000000000001e-01 8.1755097755786932e-01 8.1755098473224841e-01
1.3950000000000001e-01 8.1713789221170607e-01 8.1713789945083082e-01
1.4000000000000001e-01 8.1672582638194580e-01 8.1672583368598772e-01
1.4050000000000001e-01 8.1631477565270416e-01 8.1631478302183091e-01
1.4100000000000001e-01 8.1590473563497634e-01 8.1590474306936212e-01
1.4150000000000001e-01 8.1549570196643117e-01 8.1549570946624805e-01
1.4200000000000002e-01 8.1508767031119089e-01 8.1508767787661018e-01
1.4250000000000002e-01 8.1468063635964350e-01 8.1468064399084050e-01
1.4300000000000002e-01 8.1427459582823147e-01 8.1427460352537839e-01
1.4350000000000002e-01 8.1386954445925641e-01 8.1386955222252277e-01
1.4400000000000002e-01 8.1346547802067071e-01 8.1346548585023781e-01
1.4450000000000002e-01 8.1306239230589794e-01 8.1306240020193710e-01
1.4499999999999999e-01 8.1266028313362548e-01 8.1266029109631377e-01
1.4549999999999999e-01 8.1225914634761986e-01 8.1225915437712992e-01
1.4599999999999999e-01 8.1185897781652172e-01 8.1185898591303507e-01
1.4649999999999999e-01 8.1145977343367715e-01 8.1145978159736987e-01
1.4699999999999999e-01 8.1106152911693408e-01 8.1106153734798270e-01
1.4749999999999999e-01 8.1066424080846045e-01 8.1066424910704449e-01
1.4799999999999999e-01 8.1026790447456676e-01 8.1026791284086552e-01
1.4849999999999999e-01 8.0987251610551125e-01 8.0987252453970504e-01
1.4899999999999999e-01 8.0947807171533004e-01 8.0947808021759826e-01
1.4949999999999999e-01 8.0908456734165279e-01 8.0908457591217664e-01
1.4999999999999999e-01 8.0869199904552647e-01 8.0869200768449012e-01
1.5049999999999999e-01 8.0830036291124141e-01 8.0830037161882162e-01
1.5100000000000000e-01 8.0790965504614431e-01 8.0790966382252793e-01
1.5150000000000000e-01 8.0751987158048810e-01 8.0751988042585809e-01
1.5200000000000000e-01 8.0713100866724441e-01 8.0713101758178252e-01
1.5250000000000000e-01 8.0674306248193084e-01 8.0674307146582314e-01
1.5300000000000000e-01 8.0635602922245730e-01 8.0635603827588742e-01
1.5350000000000000e-01 8.0596990510894795e-01 8.0596991423210396e-01
1.5400000000000000e-01 8.0558468638358083e-01 8.0558469557664680e-01
1.5450000000000000e-01 8.0520036931041916e-01 8.0520037857358251e-01
1.5500000000000000e-01 8.0481695017525767e-01 8.0481695950870658e-01
1.5550000000000000e-01 8.0443442528544951e-01 8.0443443468937226e-01
1.5600000000000000e-01 8.0405279096976212e-01 8.0405280044434946e-01
1.5650000000000000e-01 8.0367204357820865e-01 8.0367205312364565e-01
1.5700000000000000e-01 8.0329217948188858e-01 8.0329218909836853e-01
1.5750000000000000e-01 8.0291319507284165e-01 8.0291320476055505e-01
1.5800000000000000e-01 8.0253508676389373e-01 8.0253509652302979e-01
1.5850000000000000e-01 8.0215785098849202e-01 8.0215786081924434e-01
1.5900000000000000e-01 8.0178148420056883e-01 8.0178149410313149e-01
1.5950000000000000e-01 8.0140598287438325e-01 8.0140599284894365e-01
1.6000000000000000e-01 8.0103134350437155e-01 8.0103135355113053e-01
1.6050000000000000e-01 8.0065756260500986e-01 8.0065757272416038e-01
1.6100000000000000e-01 8.0028463671066119e-01 8.0028464690239298e-01
1.6150000000000000e-01 7.9991256237542097e-01 7.9991257263993298e-01
1.6200000000000001e-01 7.9954133617299616e-01 7.9954134651048481e-01
1.6250000000000001e-01 7.9917095469654265e-01 7.9917096510720553e-01
1.6300000000000001e-01 7.9880141466785115e-01 7.9880142504257279e-01
1.6350000000000001e-01 7.9843271260828275e-01 7.9843272294823364e-01
1.6400000000000001e-01 7.9806484516854181e-01 7.9806485547488437e-01
1.6450000000000001e-01 7.9769780901823384e-01 7.9769781929211281e-01
1.6500000000000001e-01 7.9733160084572996e-01 7.9733161108827655e-01
1.6550000000000001e-01 7.9696621735802287e-01 7.9696622757035895e-01
1.6600000000000001e-01 7.9660165528062021e-01 7.9660166546384403e-01
1.6650000000000001e-01 7.9623791135736899e-01 7.9623792151257866e-01
1.6700000000000001e-01 7.9587498235036913e-01 7.9587499247863835e-01
1.6750000000000001e-01 7.9551286503980634e-01 7.9551287514220603e-01
1.6800000000000001e-01 7.9515155622385314e-01 7.9515156630142614e-01
1.6850000000000001e-01 7.9479105271851413e-01 7.9479106277230549e-01
1.6900000000000001e-01 7.9443135135751841e-01 7.9443136138855741e-01
1.6950000000000001e-01 7.9407244899218454e-01 7.9407245900148316e-01
1.7000000000000001e-01 7.9371434249130168e-01 7.9371435247986488e-01
1.7050000000000001e-01 7.9335702874099556e-01 7.9335703870981189e-01
1.7100000000000001e-01 7.9300050464461802e-01 7.9300051459466614e-01
1.7150000000000001e-01 7.9264476712262122e-01 7.9264477705486269e-01
1.7200000000000001e-01 7.9228981311242563e-01 7.9228982302782403e-01
1.7250000000000001e-01 7.9193563956832336e-01 7.9193564946781470e-01
1.7300000000000001e-01 7.9158224346133932e-01 7.9158225334585441e-01
1.7350000000000002e-01 7.9122962177912126e-01 7.9122963164958149e-01
1.7400000000000002e-01 7.9087777152582428e-01 7.9087778138313647e-01
1.7450000000000002e-01 7.9052668972199114e-01 7.9052669956705734e-01
1.7500000000000002e-01 7.9017637340444746e-01 7.9017638323814943e-01
1.7550000000000002e-01 7.8982681962617307e-01 7.8982682944939064e-01
1.7600000000000002e-01 7.8947802545620516e-01 7.8947803526980098e-01
1.7649999999999999e-01 7.8912998797950962e-01 7.8912999778433945e-01
1.7699999999999999e-01 7.8878270429689012e-01 7.8878271409379397e-01
1.7749999999999999e-01 7.8843617152485945e-01 7.8843618131467152e-01
1.7799999999999999e-01 7.8809038679553711e-01 7.8809039657908098e-01
1.7849999999999999e-01 7.8774534725655265e-01 7.8774535703464255e-01
1.7899999999999999e-01 7.8740105007091954e-01 7.8740105984435749e-01
1.7949999999999999e-01 7.8705749241694023e-01 7.8705750218652126e-01
1.7999999999999999e-01 7.8671467148810437e-01 7.8671468125460653e-01
1.8049999999999999e-01 7.8637258449297054e-01 7.8637259425716777e-01
1.8099999999999999e-01 7.8603122865507369e-01 7.8603123841773281e-01
1.8149999999999999e-01 7.8569060121282397e-01 7.8569061097469339e-01
1.8200000000000000e-01 7.8535069941939106e-01 7.8535070918121574e-01
1.8250000000000000e-01 7.8501152054261714e-01 7.8501153030513715e-01
1.8300000000000000e-01 7.8467306186491503e-01 7.8467307162885447e-01
1.8350000000000000e-01 7.8433532068315914e-01 7.8433533044923454e-01
1.8400000000000000e-01 7.8399829430859880e-01 7.8399830407751980e-01
1.8450000000000000e-01 7.8366198006675514e-01 7.8366198983921631e-01
1.8500000000000000e-01 7.8332637529731663e-01 7.8332638507401076e-01
1.8550000000000000e-01 7.8299147735405628e-01 7.8299148713566680e-01
1.8600000000000000e-01 7.8265728360473319e-01 7.8265729339192691e-01
1.8650000000000000e-01 7.8232379143098429e-01 7.8232380122443057e-01
1.8700000000000000e-01 7.8199099822825047e-01 7.8199100802860544e-01
1.8750000000000000e-01 7.8165890140567529e-01 7.8165891121358400e-01
1.8800000000000000e-01 7.8132749838600157e-01 7.8132750820210395e-01
1.8850000000000000e-01 7.8099678660549821e-01 7.8099679643042963e-01
1.8900000000000000e-01 7.8066676351386366e-01 7.8066677334824186e-01
1.8950000000000000e-01 7.8033742657412031e-01 7.8033743641856268e-01
1.9000000000000000e-01 7.8000877326255391e-01 7.8000878311766342e-01
1.9050000000000000e-01 7.7968080106859217e-01 7.7968081093497477e-01
1.9100000000000000e-01 7.7935350749474952e-01 7.7935351737299574e-01
1.9150000000000000e-01 7.7902689005651549e-01 7.7902689994720664e-01
1.9200000000000000e-01 7.7870094628227826e-01 7.7870095618599322e-01
1.9250000000000000e-01 7.7837567371323824e-01 7.7837568363054555e-01
1.9300000000000000e-01 7.7805106990332307e-01 7.7805107983478483e-01
1.9350000000000001e-01 7.7772713241909719e-01 7.7772714236526819e-01
1.9400000000000001e-01 7.7740385883969343e-01 7.7740386880111823e-01
1.9450000000000001e-01 7.7708124675670587e-01 7.7708125673392658e-01
1.9500000000000001e-01 7.7675929377413067e-01 7.7675930376767988e-01
1.9550000000000001e-01 7.7643799750827691e-01 7.7643800751868119e-01
1.9600000000000001e-01 7.7611735558767037e-01 7.7611736561545530e-01
1.9650000000000001e-01 7.7579736565300450e-01 7.7579737569867657e-01
1.9700000000000001e-01 7.7547802535702914e-01 7.7547803542109306e-01
1.9750000000000001e-01 7.7515933236448697e-01 7.7515934244744167e-01
1.9800000000000001e-01 7.7484128435203325e-01 7.7484129445437533e-01
1.9850000000000001e-01 7.7452387900815955e-01 7.7452388913037284e-01
1.9900000000000001e-01 7.7420711403311115e-01 7.7420712417567217e-01
1.9950000000000001e-01 7.7389098713880966e-01 7.7389099730219735e-01
2.0000000000000001e-01 7.7357549604878895e-01 7.7357550623346671e-01
}\intervalNormData
\begin{tikzpicture}[scale=.95]
  \begin{axis}[
    scale only axis,
    width=0.8\columnwidth,
    height=0.8\columnwidth,
    xmin=0,
    xmax=0.2,
    xtick={0,0.05,0.10,0.15,0.20},
    scaled x ticks=false,
    xticklabel style={
      /pgf/number format/fixed,
      /pgf/number format/precision=2,
      /pgf/number format/fixed zerofill=false,
    },
    xlabel={Time $t$},
    ylabel={$\|U^n\|_F/\|U^0\|_F$},
    grid=major,
    legend pos=north east,
    legend cell align=left,
  ]
    \addplot[
      blue,
      thick,
      mark=none,
    ] table[
      x=time,
      y=norm_sat,
    ] {\intervalNormData};
    \addlegendentry{SAT}

    \addplot[
      red,
      densely dashed,
      thick,
      mark=none,
    ] table[
      x=time,
      y=norm_fe,
    ] {\intervalNormData};
    \addlegendentry{Full FE}
  \end{axis}
\end{tikzpicture}

    \end{minipage}
     \hfill
    \begin{minipage}[t]{.49\textwidth}
        \vspace{0pt}
        \centering
\pgfplotstableread{
tau error bound id
1.2500000000000000e-04 4.8186117523180609e-05 5.0000000000000001e-04 910
2.5000000000000001e-04 7.5911603946086852e-05 1.0000000000000000e-03 911
5.0000000000000001e-04 2.7745838398643283e-04 2.0000000000000000e-03 912
1.0000000000000000e-03 4.4261580906066567e-04 4.0000000000000001e-03 913
2.0000000000000000e-03 5.7436220152604870e-04 8.0000000000000002e-03 914
4.0000000000000001e-03 2.9633543015556308e-03 1.6000000000000000e-02 915
}\intervalDtData
\begin{tikzpicture}[scale=.95]
  \begin{loglogaxis}[
    scale only axis,
    width=0.8\columnwidth,
    height=0.8\columnwidth,
    xtick={1.25e-4,2.5e-4,5e-4,1e-3,2e-3,4e-3},
    xticklabels={
      {$1.25\cdot10^{-4}$},
      {$2.5\cdot10^{-4}$},
      {$5\cdot10^{-4}$},
      {$10^{-3}$},
      {$2\cdot10^{-3}$},
      {$4\cdot10^{-3}$}
    },
    xticklabel style={rotate=45, anchor=north east},
    xlabel={Time-step size $\tau$},
    ylabel={$\max_n \|U^n-\hat U^n\|_F$},
    grid=major,
    unbounded coords=discard,
  ]
    \addplot[
      blue,
      thick,
      mark=square*,
      mark options={draw=blue, fill=blue},
    ] table[
      x=tau,
      y=error,
    ] {\intervalDtData};

    \addplot[
      blue,
      densely dotted,
      thick,
      mark=none,
      forget plot,
    ] table[
      x=tau,
      y=bound,
    ] {\intervalDtData};
  \end{loglogaxis}
\end{tikzpicture}

    \end{minipage}

    \caption{Representative study for the 1d1v problem with fixed spatial discretization and truncation tolerances \(\varepsilon_1=C\tau\), \(\varepsilon_2=C\tau^2\), using \(C=10\). Left: Normalized $L_2$ norm for the SAT solution and for the corresponding full forward Euler solution. Both norms are nonincreasing over the complete time interval, in agreement with Theorem~\ref{thm:stability_low_rank}.
    Right: Maximum SAT--forward Euler difference over all time steps for varying time steps $\tau$, including the upper bound from Theorem~\ref{thm:euler_difference} (dotted).}
    \label{fig:interval-stability-dt}
\end{figure}

We finally investigate the dependence of the SAT accuracy on the time-step size $\tau$ for the same representative experiment $C=10$. For every value of $\tau$, the SAT solution is compared with the full forward Euler solution computed with the same time step, and we report the maximum difference over all time levels. Since the initial value is represented exactly, Theorem~\ref{thm:euler_difference} yields
\[
\max_n \|U^n-\hat U^n\|_F \le T\varepsilon_1+\frac{T}{\tau}\varepsilon_2 =2TC\tau.
\]
For $T=0.2$ and $C=10$, this gives the upper bound $4\tau$, shown in the right panel of Figure~\ref{fig:interval-stability-dt}. The measured errors remain below this bound for all considered time-step sizes and decrease as the time step is refined. While the individual refinement ratios do not exhibit a uniform asymptotic order, the results are consistent with the linear-in-$\tau$ perturbation estimate of Theorem~\ref{thm:euler_difference}. Moreover, the bound appears rather conservative, with the observed errors remaining noticeably smaller than the theoretical estimate throughout the experiment.

We note that the orthogonal projections and SVD truncations employed by the SAT method do not, in general, preserve entrywise nonnegativity or the discrete conservation properties of the underlying FV scheme. For the representative computation with $C=10$, the minimum value of~$u_h$ observed over all cells and time steps is $u_{\min}=-6.99\cdot10^{-7}$, whereas the full forward Euler solution remains nonnegative. The maximum relative deviation of the mass from its initial value is $8.02\cdot10^{-5}$ for the SAT solution, compared with $1.01\cdot10^{-14}$ for the full forward Euler solution. Thus, in this example violations of positivity and mass conservation are present but remain small. As this work focuses on $L_2$ stability and control of the low-rank truncation error, positivity- and conservation-preserving modifications are left for future work.

\subsection{Transport on a triangular-shaped spatial domain (2d2v) \label{sub:2d2v}}

In our second experiment, we consider a more challenging transport problem on a two-dimensional spatial domain with nonvanishing inflow. To demonstrate the geometric flexibility of the method, we choose a triangular spatial domain and discretize it using an unstructured triangular mesh. On the coarsest mesh, we compute both the SAT solution and the corresponding full forward Euler solution, which allows us to assess the effect of the low-rank approximation for different truncation tolerances. We then consider a sequence of uniformly refined meshes to investigate the method at larger problem sizes. On these meshes, computing the full forward Euler solution becomes prohibitively expensive. Instead, we estimate the error of the SAT solution with respect to the exact solution in the continuous $L_2$ norm.

To describe the experiment more precisely, we consider the Vlasov boundary problem~\eqref{eq:vp} on the triangular spatial domain $\Omega^{(x)}$ with vertices $(-0.5,-0.5)$, $(-0.5,0.5)$, and $(0.5,0.5)$. The velocity domain is chosen as $\Omega^{(v)}=[1,3]\times[-2,1]$, which will turn out sufficiently large that the support of the solution does not reach the velocity boundary during the simulation. We consider the time interval $[0,T]$ with $T=0.5$, use the constant electric field $\bm E=(0,1)^\top$, and prescribe zero as initial condition. The inflow data are constructed from a function $\bar u(t,\bm x,\bm v)$ solving the same equation on an unbounded spatial domain. Its initial condition is given by
\[
    \bar u_0(\bm x, \bm v) = \prod_{i=1}^2 \Psi\Bigl(\frac{x_i-x^{(0)}_i}{\sigma_x}\Bigr) \, 
        \Psi\Bigl(\frac{v_i-v^{(0)}_i}{\sigma_v}\Bigr)
\]
with $\bm x^{(0)}=(-0.8,0.2)^\top$, $\sigma_x=0.3   $, $\bm v^{(0)}=(2,0)^\top$, $\sigma_v=0.5$, and $\Psi$ denotes the bump function defined in~\eqref{eq: bump function}. Thus, the compact support of $\bar u_0$ is located outside the interior of the spatial domain $\Omega^{(x)}$ and touches its inflow boundary. This setting is similar to~\cite{uschmajew2024}. By the method of characteristics, we obtain
\[
    \bar u(t,\bm x, \bm v) = \bar u_0(\bm x - t \bm v - \frac{t^2}{2} \bm E, \bm v + t \bm E).
\]
Now this function serves as the inflow condition for the problem on the triangular domain,~i.e.
\[
    u(t, \bm x, \bm v) = \bar u(t, \bm x, \bm v), \qquad (\bm x, \bm v) \in \Gamma^-(t).
\] 
Consequently, the restriction of $\bar u$ to $\Omega^{(x)}\times\Omega^{(v)}$ is the exact solution of the problem on the triangular domain and will later serve as the reference solution for the error estimation.

At the coarsest level, we use an unstructured triangular mesh for the spatial domain $\Omega^{(x)}$ with $2\,320$ triangles. The velocity domain is discretized using a uniform Cartesian grid with $64\times 96$ cells. For the coarsest mesh, the CFL threshold defined in \eqref{eq:CFL} is $\tau_{\mathrm{CFL}} \approx 1.1\times 10^{-3}$. We choose $\tau=10^{-3}$, corresponding to a normalized CFL number of $\nu_{\mathrm{CFL}}=\tau/\tau_{\mathrm{CFL}}\approx0.91 < 1$. From one refinement level to the next, we uniformly refine both the spatial and velocity meshes and halve the time step, keeping the normalized CFL number constant.

For velocity faces, we use the upwind flux. Since the Cartesian velocity mesh has only coordinate-aligned face normals, the corresponding upwind operators admit a short tensorized representation. For the spatial faces, we use a groupwise Lax--Friedrichs flux. Such a grouping provides less dissipative stabilization than a single global Lax--Friedrichs parameter while retaining a short tensorized representation of the spatial flux operator. For this purpose, we partition the range of velocity magnitudes into $N$ intervals with endpoints
\[
    \min_{\bm v \in \Omega^{(v)}} \abs{\bm v} = v_0 < v_1 < \ldots < v_N = \max_{\bm v \in \Omega^{(v)}} \abs{\bm v},
\]
and partition the discrete phase space mesh according to 
\[
    \mathcal T = \bigcup_{k=1}^N \mathcal T_k, \qquad 
    \mathcal T_k = \mathcal T^{(x)} \times \mathcal T_{k}^{(v)}, \qquad 
    \mathcal T_{k}^{(v)} = \{ K^{(v)} \in \mathcal T^{(v)} \colon v_{k-1} < \lvert \bm v(K^{(v)}) \rvert \le v_{k} \}.
\]
Here, $\bm v(K^{(v)})$ denotes the center of the velocity cell $K^{(v)}$. For each group $\mathcal T_k$, we use a Lax--Friedrichs flux with
\[
    \alpha_k^{(x)}= \max_{K^{(v)}\in\mathcal T_k^{(v)}}\max_{\bm v\in K^{(v)}} \lvert \bm v \rvert.
\]
In the simulations, we use $N=4$ equally spaced intervals. 

To apply our method, we require the inflow to be a sum of products as in~\eqref{eq: separated g}. The exact inflow~$\bar u$ does not in general have the required space--velocity separated form. However, it factorizes as
\[
    \bar u(t, \bm x, \bm v) = \bar u_1(t, x_1, v_1) \bar u_2(t, x_2, v_2).
\]
This factorization still couples each spatial coordinate $x_i$ with the corresponding velocity coordinate~$v_i$. On the $(x_1,v_1)$ and $(x_2,v_2)$ domains, we use regular grids and truncated SVDs to compute low-rank approximations of $\bar u_1$ and $\bar u_2$. Combining these factors yields a space--velocity separated approximation $g$ of the inflow data.

\begin{table}[t]
    \centering
    \caption{Comparison of SAT and full forward Euler solutions for the coarsest 2d2v discretization and the coupled tolerance choice \(\varepsilon_1=C\tau\), \(\varepsilon_2=C\tau^2\), for different values of \(C\). Reported are the maximum SAT--forward Euler difference over all time steps, the theoretical bound \(2TC\tau\) from Theorem~\ref{thm:euler_difference}, the maximum SAT rank \(r_{\max}\), the SAT runtime \(t_{\mathrm{SAT}}\), and the runtime ratio \(t_{\mathrm{FE}}/t_{\mathrm{SAT}}\).}
    \label{tab:triangle_C}
\begingroup
\small
\setlength{\tabcolsep}{3pt}
\sisetup{
  detect-all,
  output-exponent-marker=\ensuremath{\mathrm{E}},
}
\begin{tabular}{@{}
  S[table-format=3.0]
  S[table-format=1.2e-1]
  S[table-format=1.2e-1]
  S[table-format=3.0]
  S[table-format=1.2e2]
  S[table-format=3.2]
@{}}
  \toprule
  {$C$}
  & {$\max_n\|U^n-\hat U^n\|_F$} & {theor.\ bound} & {$r_{\max}$} & {$t_{\mathrm{SAT}}\,[\mathrm{s}]$} & {$t_{\mathrm{FE}}/t_{\mathrm{SAT}}$} \\
  \midrule
  1 & 2.01e-5 & 1.00e-3 & 19 & 2.99e2 & 5.78 \\
  10 & 3.03e-4 & 1.00e-2 & 9 & 1.10e2 & 14.78 \\
  100 & 3.07e-3 & 1.00e-1 & 3 & 2.93e1 & 48.61 \\
  \bottomrule
\end{tabular}
\endgroup

\end{table}

We first study the perturbation result of Theorem~\ref{thm:euler_difference}. As in the 1d1v experiment, we investigate the SAT--forward Euler difference for different choices of the tolerance parameter $C$. We use $C=1,\,10,\,100$ and the tolerances given by~\eqref{eq:epsilon_C}. The results are summarized in Table~\ref{tab:triangle_C}. We observe that the maximum difference over all time steps remains below the theoretical bound~$2TC\tau$ for all choices of $C$. As $C$ increases, the error increases, while the maximal SAT rank $r_{\mathrm{max}}$ decreases due to the more aggressive truncation. Accordingly, the SAT runtime $t_{\mathrm{SAT}}$ decreases substantially. The measured runtime ratio $t_{\mathrm{FE}}/t_{\mathrm{SAT}}$ increases from $5.87$ for $C=1$ to $48.99$ for $C=100$, illustrating the tradeoff between approximation accuracy and computational cost.

For illustration, Figure~\ref{fig:triangle-evolution} shows the computed spatial density
\[
    \rho_h(t,\bm x) = \int_{\Omega^{(v)}} u_h(t,\bm x, \bm v) \, \mathrm d \bm v
\]
on the coarsest mesh for $C=1$ at various times. We observe that the density distribution enters the domain through the left boundary and is transported across the domain while being deflected downward by the constant electric field. Eventually, the solution leaves the domain through the lower boundary.

\begin{figure}[t]
    \centering
    \begin{subfigure}[t]{0.35\textwidth}
        \centering
        \includegraphics[width=\linewidth]{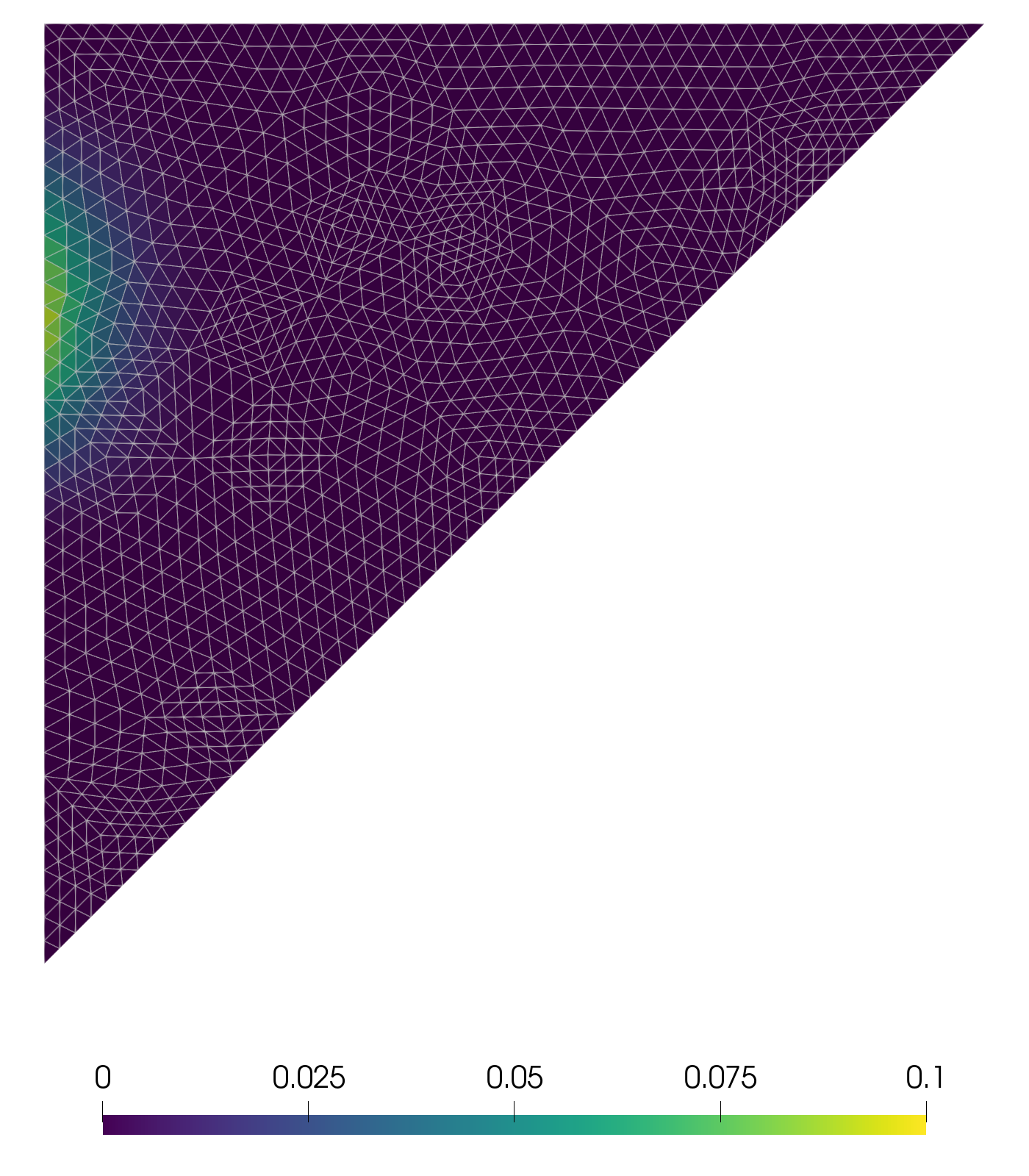}
        \caption{$t=0.1$}
        \label{fig:triangle-0p1}
    \end{subfigure}
    \qquad\qquad\qquad
    \begin{subfigure}[t]{0.35\textwidth}
        \centering
        \includegraphics[width=\linewidth]{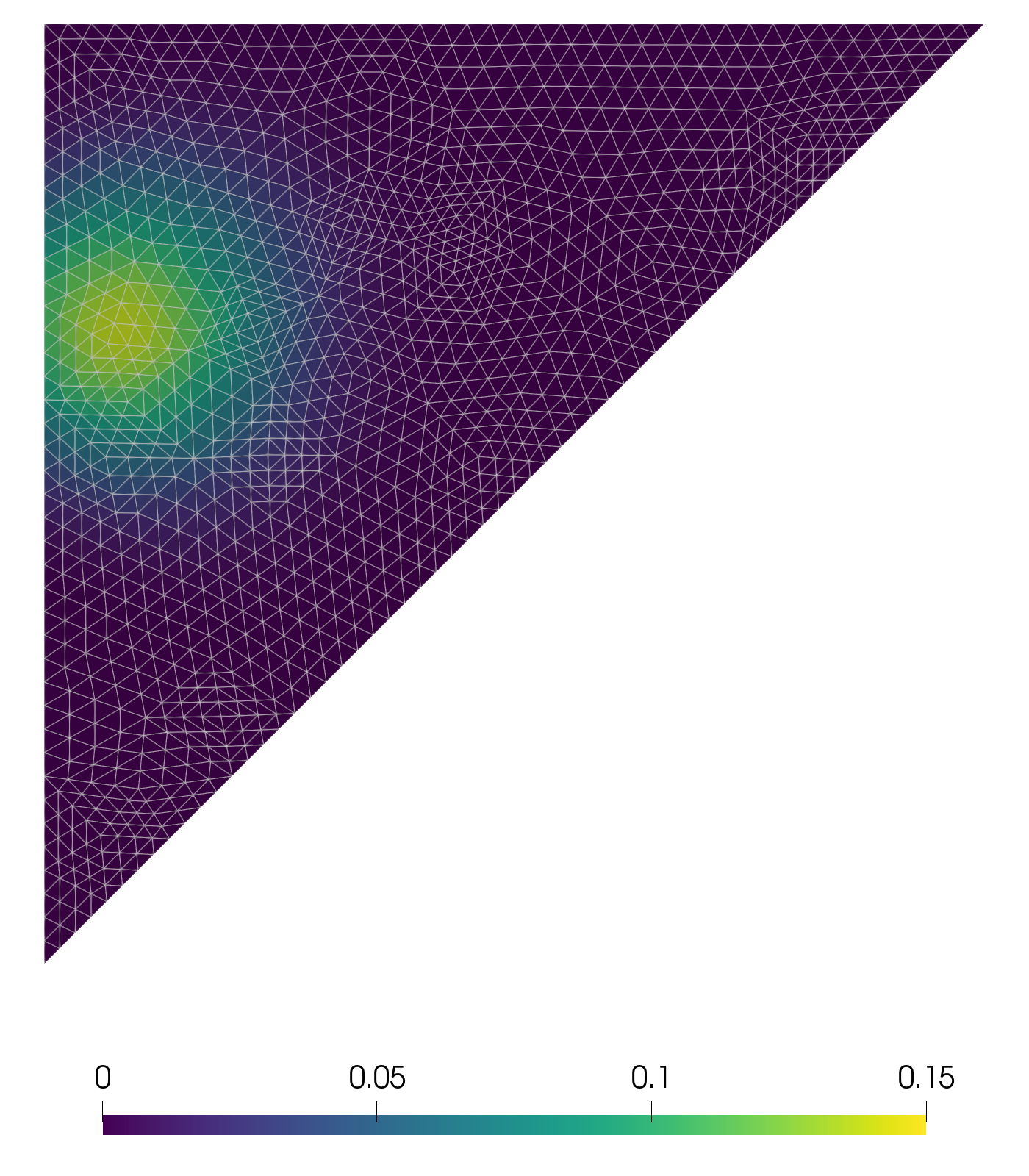}
        \caption{$t=0.2$}
        \label{fig:triangle-0p2}
    \end{subfigure}

    \medskip

    \begin{subfigure}[t]{0.35\textwidth}
        \centering
        \includegraphics[width=\linewidth]{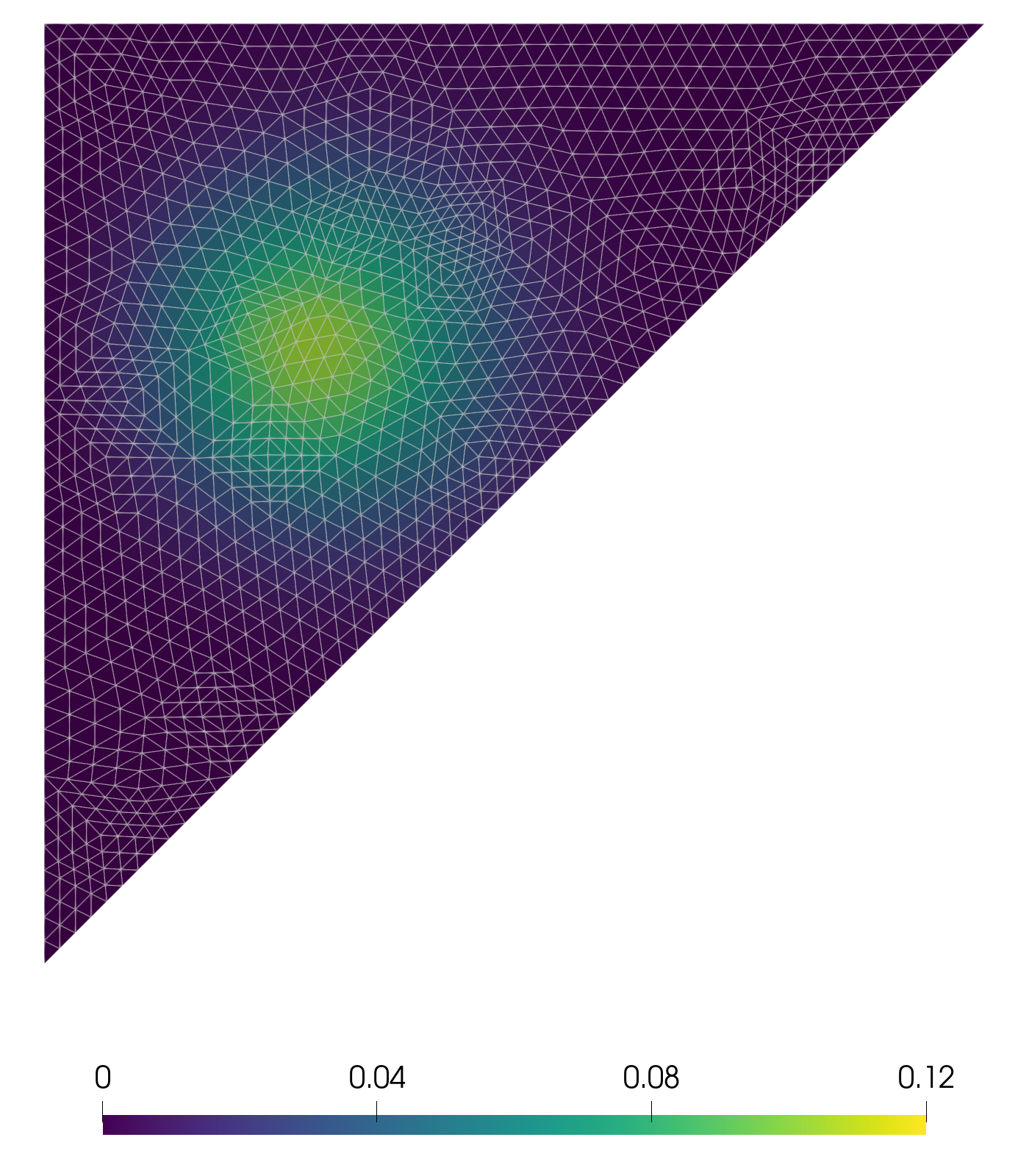}
        \caption{$t=0.3$}
        \label{fig:triangle-0p3}
    \end{subfigure}
    \qquad\qquad\qquad
    \begin{subfigure}[t]{0.35\textwidth}
        \centering
        \includegraphics[width=\linewidth]{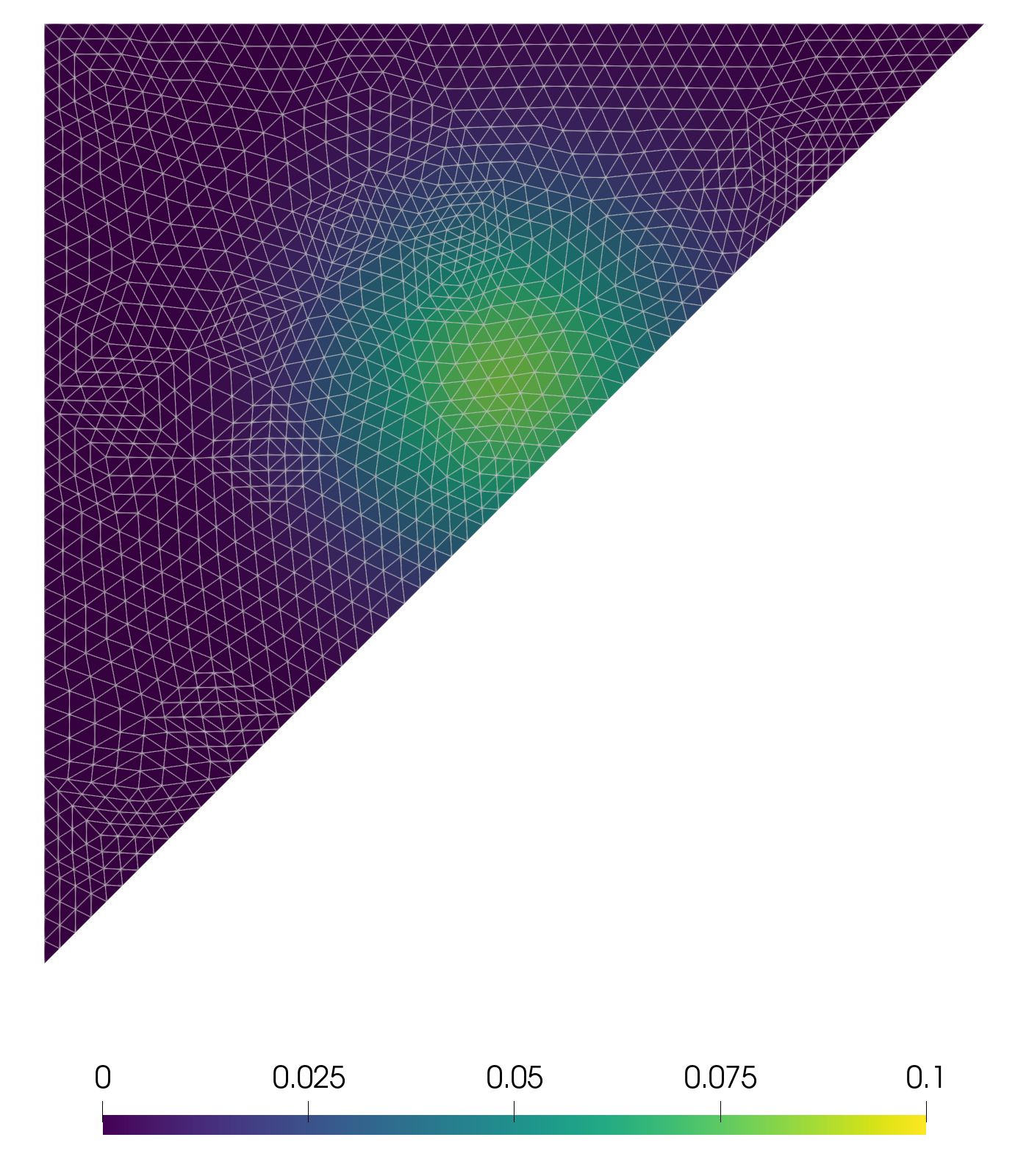}
        \caption{$t=0.4$}
        \label{fig:triangle-0p4}
    \end{subfigure}
    \caption{Spatial density $\rho_h(t,\bm x)$ on the coarsest mesh at times $t=0.1$, $0.2$, $0.3$, and $0.4$. The underlying unstructured triangular mesh is shown in each panel. The solution enters through the left boundary, is deflected downward by the electric field, and subsequently leaves the domain through the lower boundary.}
    \label{fig:triangle-evolution}
\end{figure}

The preceding experiments were designed to assess the theoretical results derived in this work. We conclude the numerical study with an additional mesh-refinement experiment that goes beyond the scope of our analysis. Its purpose is to investigate the practical behavior of the SAT method as the discretization is refined, in particular the approximation error and numerical ranks as the discretization is refined. Since a full forward Euler computation becomes prohibitively expensive on the finer meshes, we construct a reference approximation from the exact solution described above.

For this study, we use four refinement levels $\ell=0,\,1,\,2,\,3$, where the refinements are applied as described above. To estimate the $L_2$ error on the different meshes, we construct a low-rank reference approximation from the exact solution $\bar u$ using the same SVD-based procedure as for the inflow data. The resulting factors are interpolated onto the spatial and velocity meshes used for the simulation using piecewise-linear elements. We then use the difference between the numerical low-rank solution and this interpolated reference solution to estimate the $L_2$ error. The results are shown in Figure~\ref{fig:triangle-error-rank}.

In the left panel, we observe that the error initially increases as the solution enters and is transported through the domain. It reaches its largest values around $t=0.4$ and decreases again as the solution leaves the domain. The error decreases under mesh refinement, which is consistent with convergence under mesh refinement. The right panel of Figure~\ref{fig:triangle-error-rank} shows the evolution of the numerical rank for the four refinement levels. Starting from rank zero, the rank increases in several discrete steps as the solution enters and propagates through the domain. After approximately $t=0.3$, the rank remains constant for all four meshes. Because finer discretizations use smaller time steps, the truncation tolerances decrease under refinement, resulting in higher ranks. The maximum observed ranks are $3$, $6$, $9$, and $12$ for refinement levels $\ell=0$, $1$, $2$, and $3$, respectively. Thus, while the size of the full phase-space discretization grows from approximately $1.4\cdot 10^7$ to $5.8\cdot 10^{10}$ cells, the numerical rank increases only from $3$ to $12$.

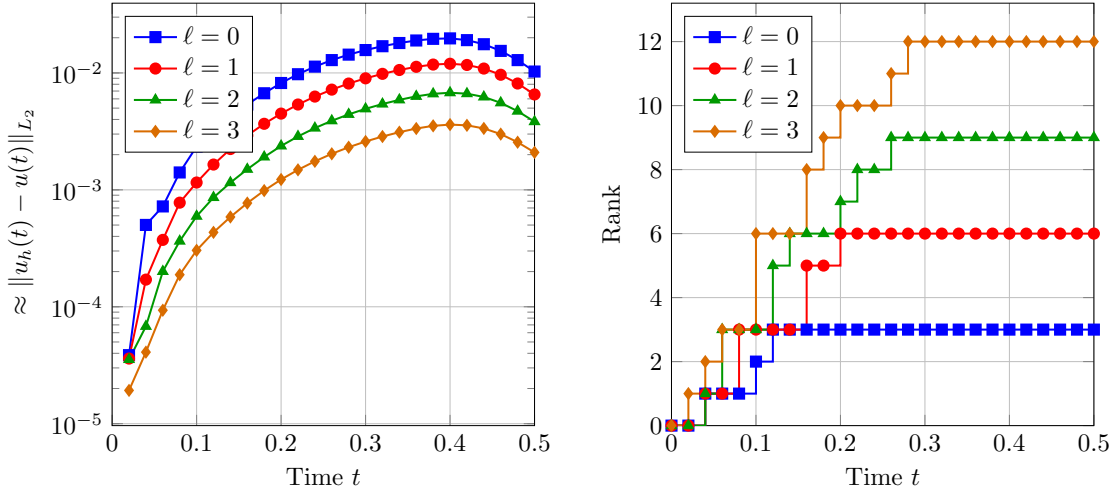
\begin{figure}[t]
    \centering

\begin{minipage}[t]{.49\textwidth}
  \vspace{0pt}
  \centering
\pgfplotstableread{
time_700 error_700 time_701 error_701 time_702 error_702 time_703 error_703
0.0000000000000000e+00 nan 0.0000000000000000e+00 nan 0.0000000000000000e+00 nan 0.0000000000000000e+00 nan
2.0000000000000000e-02 3.8593656644290301e-05 2.0000000000000000e-02 3.6077994317922446e-05 2.0000000000000000e-02 3.5399976403135120e-05 2.0000000000000000e-02 1.9266860653434559e-05
4.0000000000000001e-02 5.0124341479011649e-04 4.0000000000000001e-02 1.7076817192395282e-04 4.0000000000000001e-02 6.8025609674018275e-05 4.0000000000000001e-02 4.0935568855213965e-05
5.9999999999999998e-02 7.2158887791482157e-04 5.9999999999999998e-02 3.7338286634802496e-04 5.9999999999999998e-02 2.0046201159686582e-04 5.9999999999999998e-02 9.3420867465826950e-05
8.0000000000000002e-02 1.4087992582463625e-03 8.0000000000000002e-02 7.7717418381039024e-04 8.0000000000000002e-02 3.6433885791113674e-04 8.0000000000000002e-02 1.8811075867590190e-04
1.0000000000000001e-01 2.3469640206090110e-03 1.0000000000000001e-01 1.1555172029045018e-03 1.0000000000000001e-01 5.9486935087819926e-04 1.0000000000000001e-01 3.0391346683828271e-04
1.2000000000000000e-01 3.2409082079179934e-03 1.2000000000000000e-01 1.6459174566766765e-03 1.2000000000000000e-01 8.6005353622236886e-04 1.2000000000000000e-01 4.3296746467623979e-04
1.4000000000000001e-01 4.1772815581516826e-03 1.4000000000000001e-01 2.2341643880633023e-03 1.4000000000000001e-01 1.1529281543823597e-03 1.4000000000000001e-01 5.8628106150536625e-04
1.6000000000000000e-01 5.3686063479034853e-03 1.6000000000000000e-01 2.9192038711771851e-03 1.6000000000000000e-01 1.4992445217174280e-03 1.6000000000000000e-01 7.7089675124830529e-04
1.7999999999999999e-01 6.7019025041560034e-03 1.7999999999999999e-01 3.6723501998329718e-03 1.7999999999999999e-01 1.9112437262641313e-03 1.7999999999999999e-01 9.8578466625482233e-04
2.0000000000000001e-01 8.2010926887330100e-03 2.0000000000000001e-01 4.4898616200184918e-03 2.0000000000000001e-01 2.3739987223097259e-03 2.0000000000000001e-01 1.2263326336459388e-03
2.2000000000000000e-01 9.7495387602251106e-03 2.2000000000000000e-01 5.3688067065471639e-03 2.2000000000000000e-01 2.8645527656410217e-03 2.2000000000000000e-01 1.4832510868211756e-03
2.3999999999999999e-01 1.1319521870880598e-02 2.3999999999999999e-01 6.2774187370590153e-03 2.3999999999999999e-01 3.3754570227670327e-03 2.3999999999999999e-01 1.7537381213400473e-03
2.6000000000000001e-01 1.2874720000366225e-02 2.6000000000000001e-01 7.2003387633411704e-03 2.6000000000000001e-01 3.8977283979396682e-03 2.6000000000000001e-01 2.0336050260975893e-03
2.8000000000000003e-01 1.4338560255524475e-02 2.8000000000000003e-01 8.1135812709736549e-03 2.8000000000000003e-01 4.4207310135637636e-03 2.8000000000000003e-01 2.3152952948584768e-03
2.9999999999999999e-01 1.5689370555676239e-02 2.9999999999999999e-01 8.9921150585259447e-03 2.9999999999999999e-01 4.9310688440461900e-03 2.9999999999999999e-01 2.5915498557987260e-03
3.2000000000000001e-01 1.6922589173751838e-02 3.2000000000000001e-01 9.8213575998058675e-03 3.2000000000000001e-01 5.4203818575867908e-03 3.2000000000000001e-01 2.8584822685888498e-03
3.4000000000000002e-01 1.8025256453676964e-02 3.4000000000000002e-01 1.0594558325985513e-02 3.4000000000000002e-01 5.8842124303205194e-03 3.4000000000000002e-01 3.1136015217356951e-03
3.5999999999999999e-01 1.8946664490887253e-02 3.5999999999999999e-01 1.1282332296979327e-02 3.5999999999999999e-01 6.3042251744316749e-03 3.5999999999999999e-01 3.3462868908191810e-03
3.8000000000000000e-01 1.9570362044566410e-02 3.8000000000000000e-01 1.1794101666751509e-02 3.8000000000000000e-01 6.6272674959974541e-03 3.8000000000000000e-01 3.5281530188978616e-03
4.0000000000000002e-01 1.9746999894885618e-02 4.0000000000000002e-01 1.1982476428291892e-02 4.0000000000000002e-01 6.7688323508300289e-03 4.0000000000000002e-01 3.6152365548516069e-03
4.1999999999999998e-01 1.9077245397040914e-02 4.1999999999999998e-01 1.1699106392328708e-02 4.1999999999999998e-01 6.6473139494532428e-03 4.1999999999999998e-01 3.5638382558456421e-03
4.4000000000000000e-01 1.7602411635750877e-02 4.4000000000000000e-01 1.0900084317224969e-02 4.4000000000000000e-01 6.2294349905678030e-03 4.4000000000000000e-01 3.3524395474725027e-03
4.6000000000000002e-01 1.5434576526470538e-02 4.6000000000000002e-01 9.6474443683691319e-03 4.6000000000000002e-01 5.5514001164629944e-03 4.6000000000000002e-01 2.9980385068476803e-03
4.7999999999999998e-01 1.2870922822854290e-02 4.7999999999999998e-01 8.1210004750773852e-03 4.7999999999999998e-01 4.7098293648811908e-03 4.7999999999999998e-01 2.5547794520636592e-03
5.0000000000000000e-01 1.0277953757984227e-02 5.0000000000000000e-01 6.5469404397827937e-03 5.0000000000000000e-01 3.8314384673420162e-03 5.0000000000000000e-01 2.0910525040298681e-03
}\triangleErrorData
\begin{tikzpicture}[scale=.95]
  \begin{semilogyaxis}[
    scale only axis,
    width=0.8\columnwidth,
    height=0.8\columnwidth,
    xmin=0,
    xmax=0.5,
    xlabel={Time $t$},
    ylabel={${}\approx \| u_h(t) - u(t) \|_{L_2}$},
    grid=major,
    legend pos=north west,
    legend cell align=left,
    unbounded coords=discard,
  ]
    \addplot[
      blue,
      thick,
      mark=square*,
      mark options={draw=blue, fill=blue},
    ] table[
      x=time_700,
      y=error_700,
    ] {\triangleErrorData};
    \addlegendentry{$\ell=0$}

    \addplot[
      red,
      thick,
      mark=*,
      mark options={draw=red, fill=red},
    ] table[
      x=time_701,
      y=error_701,
    ] {\triangleErrorData};
    \addlegendentry{$\ell=1$}

    \addplot[
      green!60!black,
      thick,
      mark=triangle*,
      mark options={draw=green!60!black, fill=green!60!black},
    ] table[
      x=time_702,
      y=error_702,
    ] {\triangleErrorData};
    \addlegendentry{$\ell=2$}

    \addplot[
      orange!85!black,
      thick,
      mark=diamond*,
      mark options={draw=orange!85!black, fill=orange!85!black},
    ] table[
      x=time_703,
      y=error_703,
    ] {\triangleErrorData};
    \addlegendentry{$\ell=3$}
  \end{semilogyaxis}
\end{tikzpicture}

\end{minipage}
\hfill
\begin{minipage}[t]{.49\textwidth}
  \vspace{0pt}
  \centering
\pgfplotstableread{
time_700 rank_700 time_701 rank_701 time_702 rank_702 time_703 rank_703
0.0000000000000000e+00 0 0.0000000000000000e+00 0 0.0000000000000000e+00 0 0.0000000000000000e+00 0
2.0000000000000000e-02 0 2.0000000000000000e-02 0 2.0000000000000000e-02 0 2.0000000000000000e-02 1
4.0000000000000001e-02 1 4.0000000000000001e-02 1 4.0000000000000001e-02 1 4.0000000000000001e-02 2
5.9999999999999998e-02 1 5.9999999999999998e-02 1 5.9999999999999998e-02 3 5.9999999999999998e-02 3
8.0000000000000002e-02 1 8.0000000000000002e-02 3 8.0000000000000002e-02 3 8.0000000000000002e-02 3
1.0000000000000001e-01 2 1.0000000000000001e-01 3 1.0000000000000001e-01 3 1.0000000000000001e-01 6
1.2000000000000000e-01 3 1.2000000000000000e-01 3 1.2000000000000000e-01 5 1.2000000000000000e-01 6
1.4000000000000001e-01 3 1.4000000000000001e-01 3 1.4000000000000001e-01 6 1.4000000000000001e-01 6
1.6000000000000000e-01 3 1.6000000000000000e-01 5 1.6000000000000000e-01 6 1.6000000000000000e-01 8
1.7999999999999999e-01 3 1.7999999999999999e-01 5 1.7999999999999999e-01 6 1.7999999999999999e-01 9
2.0000000000000001e-01 3 2.0000000000000001e-01 6 2.0000000000000001e-01 7 2.0000000000000001e-01 10
2.2000000000000000e-01 3 2.2000000000000000e-01 6 2.2000000000000000e-01 8 2.2000000000000000e-01 10
2.3999999999999999e-01 3 2.3999999999999999e-01 6 2.3999999999999999e-01 8 2.3999999999999999e-01 10
2.6000000000000001e-01 3 2.6000000000000001e-01 6 2.6000000000000001e-01 9 2.6000000000000001e-01 11
2.8000000000000003e-01 3 2.8000000000000003e-01 6 2.8000000000000003e-01 9 2.8000000000000003e-01 12
2.9999999999999999e-01 3 2.9999999999999999e-01 6 2.9999999999999999e-01 9 2.9999999999999999e-01 12
3.2000000000000001e-01 3 3.2000000000000001e-01 6 3.2000000000000001e-01 9 3.2000000000000001e-01 12
3.4000000000000002e-01 3 3.4000000000000002e-01 6 3.4000000000000002e-01 9 3.4000000000000002e-01 12
3.5999999999999999e-01 3 3.5999999999999999e-01 6 3.5999999999999999e-01 9 3.5999999999999999e-01 12
3.8000000000000000e-01 3 3.8000000000000000e-01 6 3.8000000000000000e-01 9 3.8000000000000000e-01 12
4.0000000000000002e-01 3 4.0000000000000002e-01 6 4.0000000000000002e-01 9 4.0000000000000002e-01 12
4.1999999999999998e-01 3 4.1999999999999998e-01 6 4.1999999999999998e-01 9 4.1999999999999998e-01 12
4.4000000000000000e-01 3 4.4000000000000000e-01 6 4.4000000000000000e-01 9 4.4000000000000000e-01 12
4.6000000000000002e-01 3 4.6000000000000002e-01 6 4.6000000000000002e-01 9 4.6000000000000002e-01 12
4.7999999999999998e-01 3 4.7999999999999998e-01 6 4.7999999999999998e-01 9 4.7999999999999998e-01 12
5.0000000000000000e-01 3 5.0000000000000000e-01 6 5.0000000000000000e-01 9 5.0000000000000000e-01 12
}\triangleRankData
\begin{tikzpicture}[scale=.95]
  \begin{axis}[
    scale only axis,
    width=0.8\columnwidth,
    height=0.8\columnwidth,
    xmin=0,
    xmax=0.5,
    ymin=0,
    xlabel={Time $t$},
    ylabel={Rank},
    grid=major,
    legend pos=north west,
    legend cell align=left,
  ]
    \addplot[
      blue,
      thick,
      const plot,
      mark=square*,
      mark options={draw=blue, fill=blue},
    ] table[
      x=time_700,
      y=rank_700,
    ] {\triangleRankData};
    \addlegendentry{$\ell=0$}

    \addplot[
      red,
      thick,
      const plot,
      mark=*,
      mark options={draw=red, fill=red},
    ] table[
      x=time_701,
      y=rank_701,
    ] {\triangleRankData};
    \addlegendentry{$\ell=1$}

    \addplot[
      green!60!black,
      thick,
      const plot,
      mark=triangle*,
      mark options={draw=green!60!black, fill=green!60!black},
    ] table[
      x=time_702,
      y=rank_702,
    ] {\triangleRankData};
    \addlegendentry{$\ell=2$}

    \addplot[
      orange!85!black,
      thick,
      const plot,
      mark=diamond*,
      mark options={draw=orange!85!black, fill=orange!85!black},
    ] table[
      x=time_703,
      y=rank_703,
    ] {\triangleRankData};
    \addlegendentry{$\ell=3$}
  \end{axis}
\end{tikzpicture}

\end{minipage}
\caption{Simulation results for the 2d2v experiment on the triangular spatial domain. The refinement level $\ell$ denotes the number of uniform refinements relative to the coarsest mesh. Left: Estimated $L_2$ error as a function of time. Right: Numerical rank as a function of time.}
    \label{fig:triangle-error-rank}
\end{figure}

The results of the mesh-refinement simulations are summarized in Table~\ref{tab:triangle-results}. The reported errors are the maximum estimated $L_2$ errors over the 25 equally spaced checkpoints shown in Figure~\ref{fig:triangle-error-rank}.

\begin{table}[t]
    \centering
    \caption{Discretization parameters and results for the 2d2v mesh-refinement study. Here, $n=n^{(x)}n^{(v)}$ denotes the total number of phase-space cells, and $r_{\mathrm{max}}$ gives the maximal numerical rank. The estimated $L_2$ error is the maximum over the 25 equally spaced checkpoints.}
    \label{tab:triangle-results}
\begingroup
\small
\setlength{\tabcolsep}{3pt}
\sisetup{
  detect-all,
  group-minimum-digits=4,
  group-separator={\,},
  output-exponent-marker=\ensuremath{\mathrm{E}},
}
\begin{tabular}{@{}
  c
  S[table-format=1.2e-1]
  S[table-format=6.0]
  S[table-format=6.0]
  S[table-format=1.2e2]
  S[table-format=1.2e-1]
  S[table-format=1.2e-1]
  S[table-format=2.0]
  S[table-format=-1.1e-1]
@{}}
  \toprule
  Level $\ell$
  & {$\tau$}
  & {$n^{(x)}$}
  & {$n^{(v)}$}
  & {$n$}
  & {$\varepsilon_1$}
  & {$\varepsilon_2$}
  & {$r_{\mathrm{max}}$}
  & {$\max_n \|u_h^n-u(t^n)\|_{L_2(\Omega)}$} \\
  \midrule
  0 & 1.00e-3 & 2320 & 6144 & 1.43e7 & 1.00e-1 & 1.00e-4 & 3 & 2.0e-2 \\
  1 & 5.00e-4 & 9280 & 24576 & 2.28e8 & 5.00e-2 & 2.50e-5 & 6 & 1.2e-2 \\
  2 & 2.50e-4 & 37120 & 98304 & 3.65e9 & 2.50e-2 & 6.25e-6 & 9 & 6.8e-3 \\
  3 & 1.25e-4 & 148480 & 393216 & 5.84e10 & 1.25e-2 & 1.56e-6 & 12 & 3.6e-3 \\
  \bottomrule
\end{tabular}
\endgroup

\end{table}

Overall, the refinement study indicates that the SAT method remains effective on substantially larger discretizations, with decreasing approximation errors and only moderate growth of the numerical ranks.

\section{Conclusion}\label{sec:conclusion}

We developed an adaptive low-rank SAT method for FV discretizations of the Vlasov equation. The scheme accommodates bounded domains, inflow boundary conditions, and unstructured meshes. We proved that the method inherits the $L_2$ stability properties of the underlying FV scheme and derived error bounds in terms of the truncation tolerances. Our numerical experiments confirmed the theoretical predictions in both 1d1v and 2d2v settings. They further demonstrate that high-dimensional problems can be treated at low numerical rank in regimes where the corresponding full forward Euler discretization would be prohibitively expensive. Taken together, these results indicate that the SAT approach can retain the stability properties of the underlying discretization while controlling the additional error introduced by low-rank truncation and substantially reducing the effective complexity of high-dimensional transport problems.

Several open questions remain for future work. The present analysis does not address conservation or positivity preservation. In addition, higher-order spatial and temporal discretizations would be desirable. The hard SVD truncation could be replaced by more efficient low-rank approximation methods, including suitable randomized algorithms, provided that they are nonexpansive in the Frobenius norm. Finally, while we focused on the linear Vlasov equation, extending the approach to the nonlinear Vlasov--Poisson system is of considerable interest.

\appendix
\section{Stability of the fully discrete forward Euler scheme}\label{appendix}

We prove the stability of the fully discrete forward Euler scheme in Definition~\ref{def:euler_fv} for a more general class of divergence-free linear transport equations. Consider for $T>0$ the hyperbolic problem
\[
    u_t + \nabla_{\bm z} \cdot (\bm a(t,\bm z) u) = 0\quad
    \textnormal{on } \quad (0,T)\times \Omega
\]
with a bounded polyhedral domain $\Omega \subset \mathbb R^{d}$. We prescribe homogeneous inflow boundary conditions. Assume that $\bm a \in C([0,T], W^{1,\infty}(\Omega,\mathbb R^{d}))$ is divergence-free that is,~$\nabla_{\bm z} \cdot \bm a(t,\cdot) = 0$ for all $t\in[0,T]$.

Let $\mathcal T$ be a polyhedral mesh of $\Omega$ and consider the FV discretization with space $\mathcal V_h$. The fully discrete forward Euler scheme with step size $\tau>0$ generates a sequence of FV functions $(\hat u_h^n) \subset \mathcal V_h$ for times $t^n = n\tau$, $n \in \mathbb N$. On the single cell $K \in \mathcal T$, the fully discrete forward Euler scheme can be written as
\[
    \abs{K} \left( \frac{\hat u^{n+1}_K - \hat u^n_K}{\tau}\right) 
        + \sum_{e\subset \partial K} \int_e  \hat f_e(t^n, \bm z,(\hat u_h^n)^-, (\hat u_h^n)^+, \bm n_{K,e}) \,\mathrm ds =  0,
\]
where $\hat u_K^n$ denotes the constant value of $\hat u_h^n \in \mathcal V_h$ on the cell $K$. Using the numerical fluxes $\hat f_e$ as described in Section~\ref{sec:FV} this leads to the equation
\begin{equation}
    \label{eq: update function value}
    \hat u_K^{n+1} = \hat u_K^{n} - \frac{\tau}{\abs{K}} \sum_{e\subset \partial K} \abs{e} \left( a_{K,e}^n \{\hat u_h^n\}_{e} + \frac{b_{K,e}^n}{2} [\hat u_h^n]_{e} \right),
\end{equation}
where
\[
    a_{K,e}^n = \frac{1}{\abs{e}} \int_e \bm n_{K,e} \cdot \bm a(t^n,\cdot)  \,\mathrm ds, \qquad b_{K,e}^n = \frac{1}{\abs{e}} \int_e \alpha_e(t^n, \cdot)  \,\mathrm ds.
\]
Recall the different choices~\eqref{eq:upwind_flux} and~\eqref{eq:lf_flux} of $\alpha_e$ for the upwind and Lax--Friedrichs flux, respectively. We assume that the parameters $\alpha_e$ are uniformly bounded over all faces and $t\in[0,T]$. The upwind flux is always used for boundary faces.

\begin{lemma} \label{lemma:stability_fe}
    In the setting described above, define
    \begin{equation} \label{eq:CFL}
        \tau_{\mathrm{CFL}}:= \Bigl(
            \sup_{\substack{K\in\mathcal T\\ t\in[0,T]}}
                \frac{1}{2|K|}\sum_{e\subset\partial K} \int_e \alpha_e(t,\bm z)\,\mathrm ds
        \Bigr)^{-1},
    \end{equation}
    with the convention that \(\tau_{\mathrm{CFL}}=\infty\) if the supremum vanishes. Then $\tau_{\mathrm{CFL}}\in(0,\infty]$. For every initial value $\hat u_h^0\in\mathcal V_h$ and for every step size $\tau$ with $0<\tau \le \tau_{\mathrm{CFL}}$, the fully discrete forward Euler scheme is stable in the sense that the generated sequence of FV functions $(\hat u_h^n) \subset \mathcal V_h$ satisfies
    \[
        \| \hat u_h^n \|_{L_2(\Omega)} \le \| \hat u_h^0 \|_{L_2(\Omega)}
    \]
    for all $n \in \mathbb N$ with $t^n \le T$.
\end{lemma}
The proof is based on standard techniques in the analysis of FV schemes for hyperbolic conservation laws, see for example~\cite{eymard2000}. We provide a self-contained proof for completeness.
\begin{proof}
    By the uniform boundedness of the flux parameters, there exists $\bar\alpha>0$ such that
    \[
        0\le \alpha_e(t,\bm z)\le \bar\alpha
    \]
    for all faces \(e\), \(t\in[0,T]\), and \(\bm z\in e\). Therefore,
    \[
        \sup_{\substack{K\in\mathcal T\\t\in[0,T]}}
        \frac{1}{2|K|}  \sum_{e\subset\partial K} \int_e\alpha_e(t,\bm z)\,\mathrm ds
            \le  \bar\alpha \max_{K\in\mathcal T} \frac{\sum_{e\subset\partial K}|e|}{2|K|}.
    \]
    Since the mesh is finite and $\abs{K}>0$ for every $K\in\mathcal T$, the right-hand side is finite and hence the supremum in~\eqref{eq:CFL} is finite. Its reciprocal is therefore positive whenever the supremum is nonzero. If the supremum vanishes, then $\tau_{\mathrm{CFL}}=\infty$ by convention. Hence, $\tau_{\mathrm{CFL}}\in(0,\infty]$.

    Let $\mathcal N(K)$ denote the set of cells neighboring $K$. For convenience in subsequent formulas, we associate to every $e \in \partial K \cap \partial \Omega$ a ghost cell and denote the corresponding set of ghost cells as~$\hat{\mathcal N}(K)$. The augmented neighborhood of $K$ is then defined as $\mathcal N^*(K) = \mathcal N(K) \cup \hat{\mathcal N}(K)$. For every $L \in \hat{\mathcal N}(K)$, set $\hat u_L^n= 0$. We will show that for sufficiently small time steps, $\hat u_K^{n+1}$ is a convex combination of $\hat u_K^n$ and the values of~$\hat u_L^n$ in the augmented neighborhood. In fact, the update~\eqref{eq: update function value} can be rewritten as
    \begin{equation}\label{eq: convex combination}
        \hat u_K^{n+1} 
            =\underbrace{ \left[ 1-\frac{\tau}{\abs{K}} \sum_{e\subset \partial K} \abs{e} \left(\frac{a_{K,e}^n + b_{K,e}^n}{2} \right) \right]}_{{}=p_{KK}^n} \hat u_K^n
            + \sum_{L \in \mathcal N^*(K)} \underbrace{\frac{\tau}{\abs{K}}\left[\abs{e} \left(\frac{b_{K,e}^n - a_{K,e}^n}{2}\right) \right]}_{{}=p_{KL}^n} \hat u_L^n,
    \end{equation}
    where, for each $L\in\mathcal N^*(K)$, $e$ denotes the face of $K$ associated with $L$. First observe that $p_{KL}^n \ge 0$ since for upwind flux we have
    \[
        \abs{ a_{K,e}^n } = \left\lvert \frac{1}{\abs{e}} \int_e \bm n_{K,e} \cdot \bm a(t^n,\cdot)  \,\mathrm ds \right\rvert
        \le \frac{1}{\abs{e}} \int_e \abs{ \bm n_{K,e} \cdot \bm a(t^n,\cdot)}  \,\mathrm ds = b_{K,e}^n,
    \]
    whereas for Lax--Friedrichs flux
    \[
        \abs{a_{K,e}^n}\le \max_{\bm z\in e} \abs{\bm n_{K,e}\cdot\bm a(t^n,\bm z)}\le \alpha_e(t^n) = b_{K,e}^n.
    \]
    In both cases, $b_{K,e}^n - a_{K,e}^n \ge 0$ and hence $p_{KL}^n \ge 0$. In order to show that $p_{KK}^n \ge 0$, first note that
    \begin{equation} \label{eq:an_divergence_free}
        \sum_ {e \subset \partial K} \abs{e} a_ {K,e}^n = \sum_{e\subset \partial K} \int_e \bm n_{K,e} \cdot \bm a(t^n,\cdot)  \,\mathrm ds
        = \int_K \nabla \cdot \bm a(t^n,\cdot)  \,\mathrm d\bm z = 0,
    \end{equation}
    as $\bm a$ is divergence-free. Therefore, we have
    \[
        \frac{\tau}{|K|} \sum_{e\subset\partial K}|e|\frac{a_{K,e}^n+b_{K,e}^n}{2} 
            = \frac{\tau}{|K|} \sum_{e\subset\partial K}|e|\frac{b_{K,e}^n}{2} \le 1 
    \]
    for all $K\in\mathcal T$, all $t^n\le T$. The final inequality follows from the definition of $\tau_{\mathrm{CFL}}$ and the condition $0<\tau\le\tau_{\mathrm{CFL}}$.  Consequently, $p_{KK}^n\ge0$.
    
    Finally, it holds that
    \begin{align*}
            p_{KK}^n + \sum_{L \in \mathcal N^*(K)} p_{KL}^n & = 1 - \frac{\tau}{\abs{K}} \sum_{e\subset \partial K} \abs{e} \left(\frac{a_{K,e}^n + b_{K,e}^n}{2} \right) + \frac{\tau}{\abs{K}}\sum_{e\subset \partial K} \abs{e} \left(\frac{b_{K,e}^n - a_{K,e}^n}{2}\right) \\
            &= 1 - \frac{\tau}{\abs{K}} \sum_{e\subset \partial K} \abs{e} a_{K,e}^n = 1,
    \end{align*}
    where we used \eqref{eq:an_divergence_free} in the last step.

    We have thus shown that~\eqref{eq: convex combination} represents a convex combination. By applying Jensen's inequality and using $\hat u_L^n=0$ for ghost cells, we obtain
    \[
        \abs{\hat u_K^{n+1}}^2 \le p_{KK}^n \abs{\hat u_K^n}^2 + \sum_{L \in \mathcal N^*(K)} p_{KL}^n \abs{\hat u_L^n}^2
        = p_{KK}^n \abs{\hat u_K^n}^2 + \sum_{L \in \mathcal N(K)} p_{KL}^n \abs{\hat u_L^n}^2.
    \]
    Multiplying by $\abs{K}$ and summing over all cells $K$ gives
    \begin{align}
        \sum_{K \in \mathcal T} \abs{K} |\hat u_K^{n+1}|^2 &\le \sum_{K \in \mathcal T} \abs{K} p_{KK}^n \abs{\hat u_K^n}^2 + \sum_{K \in \mathcal T} \abs{K} \sum_{L \in \mathcal N(K)} p_{KL}^n \abs{\hat u_L^n}^2 \notag \\
        &= \sum_{K \in \mathcal T} \big(\abs{K} p_{KK}^n + \sum_{L \in \mathcal N(K)} \abs{L} p_{LK}^n \big) \abs{\hat u_K^n}^2.\label{eq: L2 estimate}
    \end{align}
    It remains to show that the coefficient in parentheses in~\eqref{eq: L2 estimate} is bounded by $\abs{K}$. For $e = \partial K \cap \partial L$ with $L \in \mathcal N(K)$ we have
    \[
        \abs{L} p_{LK}^n = \tau \abs{e} \left(\frac{b_{L,e}^n - a_{L,e}^n}{2} \right) = \tau \abs{e} \left(\frac{b_{K,e}^n + a_{K,e}^n}{2}\right)
    \]
    since the definition of the numerical fluxes implies $a_{L,e}^n = - a_{K,e}^n$ and $b_{L,e}^n = b_{K,e}^n$ for every interior face shared by $K$ and $L$. Inserting this into the previous equation gives
    \[
        \abs{K} p_{KK}^n + \sum_{L \in \mathcal N(K)} \abs{L} p_{LK}^n
            = \abs{K} - \tau \sum_{e \subset \partial K \cap \partial \Omega} \abs{e} \frac{a_{K,e}^n + b_{K,e}^n}{2}
            \le \abs{K},
    \]
    where the equality follows from the definitions of $p_{KK}^n$ and $p_{LK}^n$ in~\eqref{eq: convex combination}, and the inequality follows from the use of the upwind flux on boundary faces, for which $\abs{a_{K,e}^n} \le b_{K,e}^n$, and thus $a_{K,e}^n+b_{K,e}^n\ge 0$.

    As a result,~\eqref{eq: L2 estimate} yields
    \[
        \| \hat u_h^{n+1}\|_{L_2(\Omega)}^2 = \sum_{K \in \mathcal T} \abs{K} \lvert \hat u_K^{n+1} \rvert^2 \le \sum_{K \in \mathcal T} \abs{K} \abs{\hat u_K^n }^2 = \| \hat u_h^n\|_{L_2(\Omega)}^2,
    \]
    which concludes the proof of Lemma~\ref{lemma:stability_fe}.
\end{proof}

\paragraph*{Acknowledgements}
The work of A.U.~was supported by the Deutsche Forschungsgemeinschaft (DFG, German Research Foundation) – Projektnummer 506561557.



\small
\begin{thebibliography}{10}

\bibitem{mfem2024}
J.~Andrej, N.~Atallah, J.-P. B{\"a}cker, J.-S. Camier, D.~Copeland, V.~Dobrev,
  Y.~Dudouit, T.~Duswald, B.~Keith, D.~Kim, T.~Kolev, B.~Lazarov, K.~Mittal,
  W.~Pazner, S.~Petrides, S.~Shiraiwa, M.~Stowell, and V.~Tomov.
\newblock High-performance finite elements with {{MFEM}}.
\newblock {\em Int. J. High Perform. Comput. Appl.}, 38(5):447--467, 2024.

\bibitem{appelo2025}
D.~Appel\"o and Y.~Cheng.
\newblock Robust implicit adaptive low rank time-stepping methods for matrix
  differential equations.
\newblock {\em J. Sci. Comput.}, 102(3):Paper No. 81, 21, 2025.

\bibitem{bouche2005}
D.~Bouche, J.-M. Ghidaglia, and F.~Pascal.
\newblock Error estimate and the geometric corrector for the upwind finite
  volume method applied to the linear advection equation.
\newblock {\em SIAM J. Numer. Anal.}, 43(2):578--603, 2005.

\bibitem{boyer2012}
F.~Boyer.
\newblock Analysis of the upwind finite volume method for general initial- and
  boundary-value transport problems.
\newblock {\em IMA J. Numer. Anal.}, 32(4):1404--1439, 2012.

\bibitem{ceruti2022}
G.~Ceruti, J.~Kusch, and C.~Lubich.
\newblock A rank-adaptive robust integrator for dynamical low-rank
  approximation.
\newblock {\em BIT}, 62(4):1149--1174, 2022.

\bibitem{dolbeault2002}
J.~Dolbeault.
\newblock An introduction to kinetic equations: the {V}lasov-{P}oisson system
  and the {B}oltzmann equation.
\newblock {\em Discrete Contin. Dyn. Syst.}, 8(2):361--380, 2002.

\bibitem{einkemmer2021}
L.~Einkemmer and I.~Joseph.
\newblock A mass, momentum, and energy conservative dynamical low-rank scheme
  for the {V}lasov equation.
\newblock {\em J. Comput. Phys.}, 443:Paper No. 110495, 16, 2021.

\bibitem{einkemmer2024review}
L.~Einkemmer, K.~Kormann, J.~Kusch, R.~G. McClarren, and J.-M. Qiu.
\newblock A review of low-rank methods for time-dependent kinetic simulations.
\newblock {\em J. Comput. Phys.}, 538:Paper No. 114191, 27, 2025.

\bibitem{einkemmer2018a}
L.~Einkemmer and C.~Lubich.
\newblock A low-rank projector-splitting integrator for the {V}lasov-{P}oisson
  equation.
\newblock {\em SIAM J. Sci. Comput.}, 40(5):B1330--B1360, 2018.

\bibitem{einkemmer2019}
L.~Einkemmer and C.~Lubich.
\newblock A quasi-conservative dynamical low-rank algorithm for the {V}lasov
  equation.
\newblock {\em SIAM J. Sci. Comput.}, 41(5):B1061--B1081, 2019.

\bibitem{einkemmer2023}
L.~Einkemmer, A.~Ostermann, and C.~Scalone.
\newblock A robust and conservative dynamical low-rank algorithm.
\newblock {\em J. Comput. Phys.}, 484:Paper No. 112060, 20, 2023.

\bibitem{eymard2000}
R.~Eymard, T.~Gallou\"et, and R.~Herbin.
\newblock Finite volume methods.
\newblock In {\em Handbook of numerical analysis, {V}ol. {VII}}, volume VII of
  {\em Handb. Numer. Anal.}, pages 713--1020. North-Holland, Amsterdam, 2000.

\bibitem{galindo-olarte2026}
A.~{Galindo-Olarte}, J.~Nakao, M.~Pasha, J.-M. Qiu, and W.~Taitano.
\newblock A nodal discontinuous {Galerkin} method with rank-adaptive velocity
  space representation for the multiscale {BGK} model.
\newblock {\em arXiv:2508.16564}, 2025.

\bibitem{guo2024}
W.~Guo, J.~F. Ema, and J.-M. Qiu.
\newblock A local macroscopic conservative ({L}o{M}a{C}) low rank tensor method
  with the discontinuous {G}alerkin method for the {V}lasov dynamics.
\newblock {\em Commun. Appl. Math. Comput.}, 6(1):550--575, 2024.

\bibitem{guo2022a}
W.~Guo and J.-M. Qiu.
\newblock A low rank tensor representation of linear transport and nonlinear
  {V}lasov solutions and their associated flow maps.
\newblock {\em J. Comput. Phys.}, 458:Paper No. 111089, 24, 2022.

\bibitem{guo2024a}
W.~Guo and J.-M. Qiu.
\newblock A conservative low rank tensor method for the {V}lasov dynamics.
\newblock {\em SIAM J. Sci. Comput.}, 46(1):A232--A263, 2024.

\bibitem{guo2024b}
W.~Guo and J.-M. Qiu.
\newblock A local macroscopic conservative ({L}o{M}a{C}) low rank tensor method
  for the {V}lasov dynamics.
\newblock {\em J. Sci. Comput.}, 101(3):Paper No. 61, 24, 2024.

\bibitem{halko2011}
N.~Halko, P.~G. Martinsson, and J.~A. Tropp.
\newblock Finding structure with randomness: probabilistic algorithms for
  constructing approximate matrix decompositions.
\newblock {\em SIAM Rev.}, 53(2):217--288, 2011.

\bibitem{kieri2019}
E.~Kieri and B.~Vandereycken.
\newblock Projection methods for dynamical low-rank approximation of
  high-dimensional problems.
\newblock {\em Comput. Methods Appl. Math.}, 19(1):73--92, 2019.

\bibitem{kormann2026}
K.~Kormann, M.~Nazarov, and J.~Wen.
\newblock Positivity-preserving dynamical low-rank methods for the {Vlasov}
  equation.
\newblock {\em arXiv:2606.31662}, 2026.

\bibitem{kusch2023}
J.~Kusch, L.~Einkemmer, and G.~Ceruti.
\newblock On the stability of robust dynamical low-rank approximations for
  hyperbolic problems.
\newblock {\em SIAM J. Sci. Comput.}, 45(1):A1--A24, 2023.

\bibitem{leveque2002}
R.~J. LeVeque.
\newblock {\em Finite volume methods for hyperbolic problems}.
\newblock Cambridge University Press, Cambridge, 2002.

\bibitem{merlet2008}
B.~Merlet.
\newblock {$L^\infty$}- and {$L^2$}-error estimates for a finite volume
  approximation of linear advection.
\newblock {\em SIAM J. Numer. Anal.}, 46(1):124--150, 2007/08.

\bibitem{uschmajew2024}
A.~Uschmajew and A.~Zeiser.
\newblock Dynamical low-rank approximation of the {V}lasov-{P}oisson equation
  with piecewise linear spatial boundary.
\newblock {\em BIT}, 64(2):Paper No. 19, 26, 2024.

\bibitem{uschmajew2025}
A.~Uschmajew and A.~Zeiser.
\newblock Discontinuous {G}alerkin discretization of conservative dynamical
  low-rank approximation schemes for the {V}lasov-{P}oisson equation.
\newblock {\em BIT}, 65(4):Paper No. 43, 38, 2025.

\end{thebibliography}
\end{document}